\documentclass[11pt]{article}
\newtheorem{theorem}{Theorem}[section]
\newtheorem{lemma}[theorem]{Lemma}

\newtheorem{corollary}[theorem]{Corollary}

\newenvironment{proof}[1][Proof]{\begin{trivlist}
		\item[\hskip \labelsep {\bfseries #1}]}{\end{trivlist}}
\newenvironment{definition}[1][Definition]{\begin{trivlist}
		\item[\hskip \labelsep {\bfseries #1}]}{\end{trivlist}}

\newenvironment{remark}[1][Remark]{\begin{trivlist}
		\item[\hskip \labelsep {\bfseries #1}]}{\end{trivlist}}

\newcommand{\qed}{\nobreak \ifvmode \relax \else
	\ifdim\lastskip<1.5em \hskip-\lastskip
	\hskip1.5em plus0em minus0.5em \fi \nobreak
	\vrule height0.75em width0.5em depth0.25em\fi}
\usepackage[margin=1in,includefoot]{geometry}

\usepackage{setspace}
\usepackage{amsmath}
\usepackage{verbatim}
\usepackage{tikz}
\usetikzlibrary{calc,patterns}
\usepackage{tikz-network}
\usepackage{hyperref}
\usepackage{cancel}

\usepackage{placeins}

\usepackage[utf8]{inputenc}
\usepackage{physics}
\usepackage{blindtext}
\usepackage{todonotes}

\usepackage{graphicx}
\usepackage{float}

\usepackage{xfrac}

\usepackage{fancyhdr}

\usepackage{mathtools}
\usepackage{nccmath}

\usepackage{relsize}

\usepackage{bm}

\usepackage{titling}
\pretitle{\begin{center}\Huge\bfseries}
	\posttitle{\par\end{center}\vskip 0.5em}
\preauthor{\begin{center}\Large\ttfamily}
	\postauthor{\end{center}}
\predate{\par\large\centering}
\postdate{\par}

\title{Norm-dependent Lamperti-type MAP representations of stable processes and Brownian motions in the orthant}
\author{Andreas E. Kyprianou, Harry S. Mantelos, V\'ictor Rivero}
\date{\today}

\usepackage{bbm}

\usepackage{relsize}

\usepackage{mathtools}
\usepackage{csquotes}
\usepackage{mathrsfs}
\usepackage{amsbsy}

\usepackage{amssymb}

\usepackage{enumerate}
\usepackage{url}

\newcommand{\vertiii}[1]{{\left\vert\kern-0.25ex\left\vert\kern-0.25ex\left\vert #1 
		\right\vert\kern-0.25ex\right\vert\kern-0.25ex\right\vert}}

\allowdisplaybreaks
\begin{document}
	
	\maketitle
	\thispagestyle{empty}
	
\begin{abstract}
We start by remarking a one-to-one correspondence between self-similar Markov processes (ssMps) on a Banach space and Markov additive processes (MAPs) that is analogous to the well-known one between positive ssMps and L\'evy processes through the renowned Lamperti-transform, with the main difference that ours is norm-dependent. We then consider multidimensional self-similar Markov processes obtained by killing or by reflecting a stable process or Brownian motion in the orthant and we then fully describe the MAPs associated to them using the $L_1$-norm. Namely, we describe the MAP underlying the ssMp obtained by killing a $d$-dimensional $\alpha$-stable process when it leaves the orthant and the one obtained by reflecting it back in the orthant continuously (or by a jump); finally, we also describe the MAP underlying $d$-dimensional Brownian motion reflected in the orthant. The first three of the aforementioned examples are pure-jump, and the last is a diffusion, so their characterization is given through their L\'evy system, generator and/or through the modulated SDE that defines them, respectively. 
\end{abstract}

\section{Introduction and background}
Lamperti's transform for self-similar Markov processes establishes a bijection between positive self-similar Markov processes (pssMps) and L\'evy processes; cf. \cite{lamperti1972semi}. Its proof can be found in the lecture notes \cite{chaumont2006introduction}, as well as in Chapter 13.3 of \cite{kyprianou2014fluctuations} (in a more modern and complete presentation). Such a bijection has proved to be instrumental both in the understanding of path and distributional properties of the former class and in pushing the boundaries of the knowledge around the celebrated Wiener-Hopf factorization for Lévy processes. Particularly, it has given rise to the construction of new classes of Lévy processes for which explicit identities can be obtained. The recent book by Kyprianou and Pardo, \cite{kyprianou2022stable}, and the references therein present a thorough account of this. One work that sparked this chain of events is that of Caballero and Chaumont, \cite{caballero2006conditioned}, where the authors obtained a precise description of the Lévy processes underlying a \textit{stable process killed at its first passage time below zero}, a \textit{stable process conditioned to stay positive} and \textit{one conditioned to hit zero continuously}. Actually, one of the motivations for the current paper is the study of multidimensional versions of such pssMps, together with one other example of a pssMp obtained by reflecting a stable process in the positive half line, studied in \cite{kyprianou2014hitting}.

Later, Chaumont, Panti and Rivero in \cite{chaumont2013lamperti} extended Lamperti's result from the positive half-line to the real line, and then subsequently, Alili, Chaumont, Graczyk and Zak in \cite{alili2017inversion} to $\mathbb{R}^d$, by proving that the class of $\mathbb{R}^{d}$-valued ssMps is in bijection with Markov additive processes (MAPs). This bijection has been very helpful in the understanding of processes related to stable processes and to lay the foundations of a general fluctuation theory for MAP, extending the one for Markov additive processes where the modulator has a discrete state space. The latter have played a prominent role in, e.g., classical applied probability models for queues and dams.

In the present paper our aim is twofold: firstly, to extend the above-described bijections to cover the case of ssMps valued in a general Banach space with a given norm; and secondly, and that which our paper primarily serves for, is to offer a higher-dimensional extension of the theory from \cite{caballero2006conditioned} pertaining to the examples of stable processes and their respective underlying L\'evy processes previously mentioned. We do this through the study of various high-dimensional ssMps related to stable processes living, not in the positive half-line, but in the positive orthant of $\mathbb{R}^d$, that is $[0,\infty)^d$. This study will make evident the dependence on the choice of norm on $\mathbb{R}^d$.

To make more precise statements we next recall the formal definitions of ssMps and MAPs. We note that whenever we refer to a stochastic process as strong Markov, we mean it in the sense given in Chapter A.11 of \cite{kyprianou2022stable}. For consistency, we also remark that the definition of one-dimensional $\alpha$-stable process we abide to throughout the paper will always be that given in Definition 3.1 of \cite{kyprianou2022stable}.

Hereafter, $E$ will denote a locally compact and separable Banach space over the field of reals, equipped with a given norm $\|\cdot\|.$ 

\begin{definition}
	\label{def_ssMp}
An $E$-valued strong Markov process $Z=(Z_t)_{t\geq 0}$ defined on some (filtered) probability space $(\Omega,\mathcal{F},(\mathcal{F}_t)_{t\geq 0},\mathbb{P})$, with probabilities $(\mathbb{P}_x, x\in E)$, is said to be a \textit{self-similar Markov process} if it satisfies the \textit{self-similarity property}: there exists an $\alpha>0$ such that for every $c>0$ and $x\in E$,
	\begin{equation}
		\label{ss_property}
		(\{cZ_{c^{-\alpha}t},t\geq0\},\mathbb{P}_x)\stackrel{\text{d}}{=}(\{Z_t,t\geq0\},\mathbb{P}_{cx}).
	\end{equation}
	We then say that $Z$ is a ssMp with \textit{index of self-similarity} $\alpha$.
\end{definition}

\begin{definition}
	\label{def_MAPs}
	Let $E$ be a locally compact  and separable Banach space. An $\mathbb{R}\times E$-valued regular strong Markov process $(\xi,\Xi)=(\xi_t,\Xi_t)_{t\geq0}$, living in some (filtered) probability space $(\Omega,\mathcal{G},(\mathcal{G}_t)_{t\geq0},\mathbb{P})$, with probabilities $\left(\mathbb{P}_{(x,\theta)}, (x,\theta)\in\mathbb{R}\times E\right)$, and cemetery state $(-\infty,\partial)$, where $\partial$ is a point not in $E$, is called a \textit{Markov additive process (MAP)} if $\Xi=(\Xi_t)_{t\geq0}$ is a regular strong Markov process on $E$ with cemetery state $\partial$ such that, for every bounded measurable function $f:\mathbb{R}\times E\to \mathbb{R}^{+}$, $t,s\geq0$ and $(y,\theta)\in\mathbb{R}\times E$,
	\begin{equation}
		\label{MAP_eqn2}
		\mathbb{E}_{(y,\theta)}[f(\xi_{t+s}-\xi_t,\Xi_{t+s})\mathbbm{1}_{\{t+s<\zeta\}}|\mathcal{G}_t]=\mathbb{E}_{(0,\theta')}[f(\xi_{s},\Xi_{s})\mathbbm{1}_{\{s<\zeta\}}]\Bigr|_{\substack{\theta'=\Xi_t}},\quad \mathbb{P}_{(y,\theta)}-\text{a.s.},
	\end{equation} 
	where as usual $\zeta\coloneqq \inf\{t>0:\Xi_t=\partial\}$ denotes the lifetime of $(\xi,\Xi)$. We call the process $\Xi=(\Xi_t)_{t\geq0}$ the \textit{modulator} of the MAP, and $\xi=(\xi_t)_{t\geq0}$ the \textit{ordinate}.
\end{definition}
Property (\ref{MAP_eqn2}) is arguably the defining property of MAPs. It is reminiscent of the stationary and independent increments property of L\'evy processes, and ensures that conditionally on $\Xi$ the ordinate is an additive process, that is, it has independent increments. One of its most important implications is the translation invariance property of the ordinate: 
\begin{equation}
	\label{additive_lemma}
	(\{(a+\xi_t,\Xi_t),t\geq0\},\mathbb{P}_{(x,\theta)})\stackrel{\text{d}}{=}(\{(\xi_t,\Xi_t),t\geq0\},\mathbb{P}_{(a+x,\theta)}),\qquad(x,\theta)\in\mathbb{R}\times E,\enspace a\in\mathbb{R}.
\end{equation}
The pioneering papers of Cinlar: \cite{ccinlar1972markovI}, \cite{ccinlar1972markovII}, \cite{cinlar1975exceptional}, \cite{ccinlar1973levy} and \cite{ccinlar1976entrance} are one of the main sources of information about general MAP. 

In order to state the earlier-mentioned generalization of Lamperti's transform for a general Banach space $E$ over the field of reals, equipped with a given norm $\|\cdot\|$, we recall that any element $x\in E$ possesses a (unique) \textit{$\|\cdot\|$-polar decomposition}: $x$ can be uniquely identified with the vector 
	\begin{equation}
		\label{def_norm_polar_decomp_of_vector}
		(\|x\|,\arg(x))\in[0,\infty)\times\mathcal{S}_{\|\cdot\|},
	\end{equation}
	where $\arg(x)\coloneqq\frac{1}{\|x\|}x$, and $\mathcal{S}_{\|\cdot\|}\coloneqq\{z\in E:\|z\|=1\}$ denotes the unit sphere of $E$ with respect to $\|\cdot\|$. As a convention, we take $(0,\partial)$ to be the $\|\cdot\|$-polar representation of the vector $0\in E$. 
	
Later on, we will require the notation \begin{equation}\label{L1}
\mathcal{S}_{1}^d\coloneqq \{\boldsymbol{x}\in\mathbb{R}^d:\|\boldsymbol{x}\|_{1}=1\},\qquad \mathcal{S}_{1}^{d,+}\coloneqq\mathcal{S}_{1}^d\cap (0,\infty)^d\coloneqq \{\boldsymbol{x}\in(0,\infty)^d:\|\boldsymbol{x}\|_{1}=1\}.
\end{equation}

We can now state our first main result.

\begin{theorem}
	\label{thm_norm_dep_lamperti_transform}
	Let $E$ be a locally compact  and separable Banach space over the field of reals, equipped with a norm $\|\cdot\|$. Fix $\alpha>0$. Let $(\xi_t,\Xi_t)_{t\geq0}$ be an $\mathbb{R}\times \mathcal{S}_{\|\cdot\|}$-valued MAP with cemetery state $(-\infty,\partial)$ and lifetime $\zeta$. Set $\phi(t)\coloneqq \inf\{s>0: \int_0^s {\rm e}^{\alpha\xi_r}dr>t\}$, $t\geq0$. Then, the $E$-valued stochastic process that has the $\|\cdot\|$-polar decomposition 
	\begin{equation}
		\label{thm_norm_rep_process}
		({\rm e}^{\xi_{\phi(t)}},\Xi_{\phi(t)})_{t\geq 0},
	\end{equation}
	where we understand ${\rm e}^{-\infty}=0$, is a ssMp with index of self-similarity $\alpha$, cemetery state $0\in E$ and lifetime $K=\int_0^\zeta {\rm e}^{\alpha\xi_s} ds$. 
	
	Conversely, let $(Z_t)_{t\geq0}$ be an $E$-valued ssMp with index of self-similarity $\alpha$, cemetery state $0$ and lifetime $K\coloneqq\inf\{t>0: \|Z_t\|=0\}$. The $\mathbb{R}\times\mathcal{S}_{\|\cdot\|}$-valued process $(\xi,\Xi)$, defined by the transformation
	\begin{equation}
		\label{def_of_underlying_norm_dep_MAP}
		(\xi_t,\Xi_t)=\begin{cases}
			(\log(\|Z_{I_t}\|),\frac{1}{\|Z_{I_t}\|}Z_{I_t}), & \text{if $t<\zeta$}\\
			(-\infty,\partial), & \text{if $t\geq \zeta$};
		\end{cases}
	\end{equation}
	where $I_t\coloneqq\inf\{s>0:\int_0^s \|Z_r\|^{-\alpha} dr>t\}$, $t\geq 0;$  $\zeta:=\int^{K}_0 \|Z_r\|^{-\alpha} dr.$  
	is a MAP with cemetery state $(-\infty,\partial)$ and lifetime $\zeta$. We call $(\xi_t,\Xi_t)_{t\geq0}$ the underlying MAP of the ssMp $(Z_t)_{t\geq0}$.
\end{theorem}

We reference a version of this result, applied in a different way and in a much different context, as it occurred that we had happened to work on this type of generalization of the Lamperti transform at the same time, in the recent paper of Siri-J\'egousse and Wences, \cite{siri2024lamperti}.

In this paper, we are interested in using the general description given in the converse of Theorem \ref{thm_norm_dep_lamperti_transform} to describe the underlying MAPs of four ssMps, where $E$ therein will be taken to be the positive orthant in $\mathbb{R}^d$, and $\norm{\cdot} = \norm{\cdot}_1$, the $L_1$-norm on this space. The special attention we give to the case of $\|\cdot\|=\|\cdot\|_{1}$ brings about especially aesthetic closed-form formulae and it contrasts with the usual $L_2$ norm. Of course the results can be given in the more general setting of any $L_p$-norm (or any other norm, for that matter), but for the sake of brevity we provide no further details on this in this paper. For further details we invite the interested reader to consult the PhD-thesis of one of the authors of this paper, HSM.

The four stochastic processes we consider are:

\begin{enumerate}[(a)]
	\item A \textit{$d$-dimensional $\alpha$-stable process killed upon exiting the orthant} taking the form $Z=(Z_t)_{t\geq0}$,
		\begin{equation}
			\label{item_a}
		Z_t=(X^{(1)}_t,\ldots,X^{(d)}_t)\mathbbm{1}_{\{t<\tau^D\}},\enspace\tau^D\coloneqq\inf\{t:X^{(i)}_t<0\text{ for some $1\leq i\leq d$}\},\qquad t\geq0,
		\end{equation}
	where the $X^{(i)}=(X^{(i)}_t)_{t\geq0}$ are independent one-dimensional $\alpha$-stable processes with both positive and negative jumps.
	\item \textit{A reflected $d$-dimensional symmetric $\alpha$-stable process} taking the form $\hat{X}=(\hat{X}_t)_{t\geq0}$,
		\begin{equation}
			\label{item_b}
		\hat{X}_t=(|X^{(1)}_t|,\ldots,|X^{(d)}_t|),\qquad t\geq0,
		\end{equation}
	where the $X^{(i)}=(X^{(i)}_t)_{t\geq0}$ are iid one-dimensional symmetric $\alpha$-stable processes, killed and absorbed at the first time $\hat{X}$ hits the zero vector $\boldsymbol{0}_d\in\mathbb{R}^d$.
	\item A \textit{Skorokhod-reflected $d$-dimensional $\alpha$-stable process} taking the form $R=(R_t)_{t\geq0}$,
		\begin{equation}
			\label{item_c}
		R_t=(X_t^{(1)}-(0\land\underbar{X}_t^{(1)}),\ldots,X_t^{(d)}-(0\land\underbar{X}_t^{(d)})),\qquad t\geq0,
		\end{equation}
	where the $X^{(i)}=(X^{(i)}_t)_{t\geq0}$ are iid one-dimensional spectrally-positive $\alpha$-stable processes, $$\underbar{X}_t^{(i)}\coloneqq\inf_{s\leq t}X_s^{(i)};$$ and we assume the process $R$ is killed and absorbed at the first time it hits the zero vector $\boldsymbol{0}_d\in\mathbb{R}^{d}$.
	\item A \textit{Skorokhod-reflected $d$-dimensional Brownian motion} taking the form $\mathcal{R}=(\mathcal{R}_t)_{t\geq0}$,
		\begin{equation}
			\label{item_d}
		\mathcal{R}_t=(B_t^{(1)}-(0\land\underbar{B}_t^{(1)}),\ldots,B_t^{(d)}-(0\land\underbar{B}_t^{(d)})),\qquad t\geq0,
		\end{equation}
	where the $B^{(i)}=(B^{(i)}_t)_{t\geq0}$ are iid (one-dimensional) Brownian motions, and $$\underbar{B}_t^{(i)}\coloneqq\inf_{s\leq t}B_s^{(i)},$$ where we assume that $\mathcal{R}$ is killed and absorbed at its first hitting time of $\boldsymbol{0}_d\in\mathbb{R}^{d}.$
\end{enumerate}

By an $\alpha$-stable process, $(X_t)_{t\geq0}$, we mean  the L\'evy process with jump measure $\Pi$ taking the form
\begin{equation}
	\label{def_levy_measures_of_1dim_indep_stables}
	\Pi(dx)=|x|^{-(1+\alpha)}(c_1\mathbbm{1}_{\{x>0\}}+c_2\mathbbm{1}_{\{x<0\}})dx,\qquad x\in\mathbb{R},
\end{equation}
where 
\begin{equation}
c_1\coloneqq \frac{\Gamma(1+\alpha)\sin(\pi\alpha\rho)}{\pi} \text{ and }c_2\coloneqq \frac{\Gamma(1+\alpha)\sin(\pi\alpha(1-\rho))}{\pi}
\label{c-s}
\end{equation}
Moreover, its  characteristic exponent   satisfies
\begin{equation}
	\label{char_exp_of_1d_stable}
	\Psi (z)%=\int_{\mathbb{R}} (1-{\rm e}^{-i\lambda x}+i\lambda x\mathbbm{1}_{\{|x|<1\}})\Pi (dx) 
	=  |z|^{\alpha} \left({\rm e}^{\pi i \alpha (\frac{1}{2}-\rho)} {\bf 1}_{(z>0)}+ 
 {\rm e}^{-\pi i \alpha (\frac{1}{2}-\rho)} {\bf 1}_{(z<0)}\right), \qquad z\in\mathbb{R}.
\end{equation}
The setting of positive and negative jumps conforms to the parameter regime $\rho\in(0,1)$. Symmetric $\alpha$-stable processes correspond to the setting $\rho = 1/2$. The spectrally positive setting conforms to the parameter regime that $\alpha\in(1,2)$ and $\alpha(1-\rho) = 1$.

{In each of the four cases (a)-(d) above, establishing that the processes $Z$, $\hat{X}$, $R$ and $\mathcal{R}$ are self-similar Markov processes is a straightforward variant of their one dimensional versions.}

\begin{comment}{The choice of $\|\cdot\|$ alters the geometry of the unit sphere of $\mathbb{R}^d$ in which the underlying modulator lives in; e.g., for $\|\cdot\|=\|\cdot\|_{1}$ it is the simplex; for $\|\cdot\|=\|\cdot\|_{2}$ it's the usual (Euclidean) unit sphere. The special attention we give to the case of $\|\cdot\|=\|\cdot\|_{1}$ brings about especially aesthetic closed-form formulae. For this reason, we have chosen to present here in the introduction some of the dimension $d=2$ versions, thereby giving some flavour of the main body of the paper. }\end{comment}

Our description of the underlying MAPs of the above-listed ssMps is given by means of their infinitesimal generator. In the case of the processes arising from (a) and (b), one will see that  that their infinitesimal generator has no local part, and hence are totally determined by their inherent killing rate and respective jump structures, which in turn are totally determined by their respective L\'evy system, of which we adopt the definitions and terminology given in \cite{kyprianou2025stronglawlargenumbers} and \cite{ccinlar1973levy}, which we briefly recall next.

\begin{definition}
	Let $(\xi,\Theta)=(\xi_t,\Theta_t)_{t\geq 0}$ be a $\mathbb{R}\times E$-valued MAP. By a {\it L\'evy system} for $(\xi,\Theta)$ we mean a couple $(H, L)$, where $H$ is an additive functional $t\mapsto H_t$, and  $L$ is a kernel from $(E,\mathcal{E})$ to $(E\times\mathbb{R},\mathcal{E}\bigotimes\mathcal{B}(\mathbb{R})),$ such that for every bounded measurable function $f:\mathbb{R}^{2}\times E^{2}\to[0,\infty)$, $\boldsymbol{\theta}\in E$ and $t\geq0$,
	\begin{equation}
	\begin{split}
		\label{def_levy_systems_for_MAPs_ver3}
		&\mathbb{E}_{(0,\boldsymbol{\theta})}\Bigg[\sum_{s\leq t}f(\xi_{s-},\Delta \xi_s, \Theta_{s-}, \Theta_{s}) \mathbbm{1}_{\{\Delta \xi_s \neq 0\}}\Bigg]\\
		&=\mathbb{E}_{(0,\boldsymbol{\theta})}\Bigg[\int_0^t \int_{\mathbb{R}\times E} f(\xi_{s},y,\Theta_s,\boldsymbol{\phi}) L(\Theta_s,d\boldsymbol{\phi},dy)dH_s\Bigg].
		\end{split}
	\end{equation}
 We refer to the transition kernel $L$ as the \textit{jump kernel} of the MAP.
\end{definition}

As remarked in \cite{kyprianou2025stronglawlargenumbers}, in most examples of MAP appearing in the literature, the additive functional $H$ has the form $H_t=t\wedge \zeta$, where $\zeta$ denotes the lifetime of the MAP. In fact, this also happens to be the case with all the MAPs we characterize in the paper as well. As such, we will always omit writing about the L\'evy system of the respective MAP, and solely refer to its jump kernel (hence it will be understood that $H$ is of the above form). 

In the case of the processes arising in (c) and (d), their respective underlying MAPs will be described through their generators; and in the special case of dimension $d=2$, in terms of a solution to a modulated SDE. While describing MAPs associated to transformations of stable processes via their jump structure, at least in dimension one, has received quite a bit of attention in recent years, the case of MAPs that involve a boundary condition (like the one we will study arising from the ssMp described in (c)) is more novel and believe will strike interest in pursuing their understanding. 

\begin{lemma}
	\label{compact_thm_killing_rate}
 Let $(\xi,\Xi)=(\xi_t,\Xi_t)_{t\geq0}$ be the underlying MAP of $Z$ from (\ref{item_a}) with respect to $\|\cdot\|_1$. The killing rate of $(\xi,\Xi)$ is given by
	\begin{equation}
		q(\boldsymbol{x})=\mathbbm{1}_{(0,\infty)^d}(\boldsymbol{x})\frac{1}{\alpha}\sum_{k=1}^d \frac{\Gamma(1+\alpha)\sin(\pi\alpha(1-\rho_k))}{\pi} \Big(\frac{x_k}{\|\boldsymbol{x}\|_1}\Big)^{-\alpha},\quad \boldsymbol{x}=(x_1,\ldots,x_d)\in(0,\infty)^d,
		\label{q}
	\end{equation}
where $\rho_k\coloneqq\mathbb{P}(X^{(k)}_1>0)$ is the positivity parameter of $X^{(k)}$ from the description in  (\ref{item_a}).
\end{lemma}

Below is a result that describes the jumps of the underlying MAP, $(\xi,\Xi)$, of the ssMp $Z$ from (\ref{item_a}) with respect to the $L^1$ norm in dimension $d>1$. In order to state the result, we introduce the following notation for a particular vector in $\mathbb{R}^d$ (which makes a frequent appearance in our paper): for $y\in\mathbb{R},\enspace1\leq j \leq d$, and $\boldsymbol{\theta}=(\theta^{(1)},\ldots,\theta^{(d)})\in\mathbb{R}^d$, we define the vector
\begin{equation}
	\label{frequent_vector}
	\boldsymbol{v}^{\boldsymbol{\theta}}_{j,y}\coloneqq {\rm e}^{-y}(\theta^{(j)}-1+{\rm e}^{y})\boldsymbol{e}_j+\sum_{k\neq j} {\rm e}^{-y}\theta^{(k)}\boldsymbol{e}_k,
\end{equation}
where as usual, $\{\boldsymbol{e}_j:1\leq j\leq d\}$ denotes the standard orthonormal basis of $\mathbb{R}^d$. In particular, in the case of $d=2$, one gets
	\begin{equation*}
		\boldsymbol{v}^{\boldsymbol{\theta}}_{1,y}= (1-(1-\theta^{(1)}){\rm e}^{-y},{\rm e}^{-y}\theta^{(2)})^T=(1-{\rm e}^{-y}\theta^{(2)},{\rm e}^{-y}\theta^{(2)})^T=(\theta^{(1)}+(1-{\rm e}^{-y})\theta^{(2)},{\rm e}^{-y}\theta^{(2)})^T;
	\end{equation*}
	\begin{equation*}
		\boldsymbol{v}^{\boldsymbol{\theta}}_{2,y}= ({\rm e}^{-y}\theta^{(1)},1-(1-\theta^{(2)}){\rm e}^{-y})^T=({\rm e}^{-y}\theta^{(1)},1-{\rm e}^{-y}\theta^{(1)})^T=({\rm e}^{-y}\theta^{(1)},\theta^{(2)}+(1-{\rm e}^{-y})\theta^{(1)})^T.
	\end{equation*}
We have all the elements to describe the jump kernel of our first example. 
\begin{theorem}
	\label{thm_quintuple_law_up_to_first_exit_l1_norm}
	Let $(\xi,\Xi)$ be the underlying MAP of $Z$ from (\ref{item_a}) with respect to $\|\cdot\|_1$. The jump kernel of $(\xi,\Xi)$ is given by
	\begin{equation*}
		L(\boldsymbol{\theta},d\boldsymbol{\phi},dy)=\sum_{j=1}^d \Big(c_1^{(j)}\mathbbm{1}_{(0,\infty)}(y)+c_2^{(j)}\mathbbm{1}_{\big(\log(1-\theta^{(j)}),0\big)}(y)\Big)\delta_{\boldsymbol{v}_{j,y}^{\boldsymbol{\theta}}}(d\boldsymbol{\phi})\frac{{\rm e}^y}{|{\rm e}^y-1|^{1+\alpha}}dy,	,
	\end{equation*}
for $\boldsymbol{\theta}=(\theta^{(1)},\ldots,\theta^{(d)})\in\mathcal{S}^{d,+}_{1}$, $(\boldsymbol{\phi},y)\in\mathcal{S}^{d,+}_{1}\times\mathbb{R}$, where the vector $\boldsymbol{v}^{\boldsymbol{\theta}}_{j,y}$ is as defined in (\ref{frequent_vector}), and $\delta_{\boldsymbol{x}}$, for any $\boldsymbol{x}\in\mathbb{R}^d$, denotes the classic Dirac measure on $\mathbb{R}^d.$
\end{theorem} 

Moving on to the MAP $(\hat{\xi},\hat{\Xi})=(\hat{\xi}_t,\hat{\Xi}_t)_{t\geq0}$ that underlies the ssMp $\hat{X}$ from (\ref{item_b}), its jump structure is formed by two types of jumps, those arising from $\hat{X}$ while in the orthant, which are also described by a jump kernel that equals the one in the previous result, and those that arise when the reflection takes effect. The result we present below gives the distribution of the ``corrective" jump the MAP $(\hat{\xi},\hat{\Xi})$ makes at the time of reflection to account for the effect of reflecting the (symmetric) stable process back into the orthant via a jump. 

Let us denote by $L$ the time of the first corrective jump. Note that $(\hat{\xi}_t,\hat{\Xi}_t)_{0\leq t<L}$ is equal in law to  $(\xi_t, \Xi_t)_{0\leq t< K}$, where we recall that $(\xi, \Xi) $ was the MAP describing $Z$ from (\ref{item_a}), and $K$ denotes its killing time. The first corrective jump  depends on the value of the modulator $\hat{\Xi}$ just prior to the time of the first corrective  jump, that is  $\hat{\Xi}_{\hat{L}-}$. By calculating the distribution of the first corrective jump conditional 
on $\hat{\Xi}_{L-}$, by the strong Markov property, this also gives us the conditional distribution of all subsequent corrective jumps.

This situation is analogous to the one-dimensional setting whereby the underlying L\'evy process of a one-dimensional reflected symmetric stable process was shown to be intimately connected with that of a one-dimensional symmetric stable process killed upon exiting the positive half-line (cf., \cite{kyprianou2014hitting}; or, alternatively, Proposition 5.20 in \cite{kyprianou2022stable}).
%\todo{there is an issue here, how is the time $K$ defined? even though it is the first time the stable leaves the orthant, and this is represented using Lamperti's transform, this does not make sense in terms of the underlying MAP.}
\begin{lemma}
	\label{thm_concise_ordinate_corrective_jump}
Let $(\hat{\xi},\hat{\Xi})=(\hat{\xi}_t,\hat{\Xi}_t)_{t\geq0}$ be the underlying MAP of $\hat{X}$ from (\ref{item_b}) with respect to the $L^1$ norm on $\mathbb{R}^d$, and $L$ be the first time of appearance of a jump of $(\hat{\xi},\hat{\Xi})$ due to the reflection. Let $\tau_D=\inf\{t>0: X_t\notin D\},$ the first exit from the orthant  $D:=(0,\infty)$ by the stable process $X$ in (a). The distribution of the ``corrective" jumps of the MAP $(\hat{\xi},\hat{\Xi})$ is determined by the first exit from the orthant $D$ as follows
\begin{equation}
\begin{split}
&{E}_{\theta}\left(G\left(\log\left(\frac{\|X_{\tau_D}\|_1}{\|X_{\tau_D-}\|_1}\right), \left(\frac{|X^{(i)}_{\tau_D}|}{\|X_{\tau_D}\|}, i\in\{1,\ldots, d\}\right)\right) \Bigg| X_{\tau_D-}\right)%|_{\{X_{\tau_D-}=\beta\}}
\\
&\hspace{2cm}={\mathbb{E}_{0,\theta}\left(G\left(\Delta\hat{\xi}_{L}, \hat{\Xi}_L\right)\Big| \hat{\Xi}_{L-}=X_{\tau_D-}/\norm{X_{\tau_D-}}_1 \right),%|_{\{\hat{\Xi}_{L-}=X_{\tau_D-}/\norm{X_{\tau_D-}}_1\}},
\qquad \theta\in \mathcal{S}_{\|\cdot\|_1} }
\end{split}
\end{equation}
for any bounded measurable function $G:\mathbb{R}\times\mathcal{S}^{d,+}_{1}\to\mathbb{R}$ and any $\boldsymbol{\theta}\in\mathcal{S}^{d,+}_{1}.$ The latter can be computed further via 
\begin{equation}\label{eq:234zillion}
	\mathbb{E}_{(0,\boldsymbol{\theta})}[G(\Delta \hat{\xi}_{L},\hat{\Xi}_{L})|\hat{\Xi}_{L-}]  = H(\hat{\Xi}_{L-})
%	\\
%	&=q(\hat{\Xi}_{L-})^{-1}\cdot\sum_{j=1}^d \int_{\log(1-\hat{\Xi}_{L-}^{(j)})}^\infty G\Big(x,\frac{1}{{\rm e}^x}\big(\hat{\Xi}_{L-}+({\rm e}^x-1)\boldsymbol{e}_j\big)\Big)\cdot ({\rm e}^x+2\hat{\Xi}_{L-}^{(j)}-1)^{-(1+\alpha)}\cdot {\rm e}^x dx,
\end{equation}
where, for $\Xi\in \mathcal{S}^{d,+}_{1}$
\begin{equation}
H(\Xi)=q(\Xi)^{-1}\cdot\sum_{j=1}^d \int_{\log(1- \Xi^{(j)})}^\infty G\Big(x,\frac{1}{{\rm e}^x}\big( \Xi +({\rm e}^x-1)\boldsymbol{e}_j\big)\Big)\cdot ({\rm e}^x+2 \Xi^{(j)}-1)^{-(1+\alpha)}\cdot {\rm e}^x dx,
\end{equation}
  and we are using the notation $\Xi^{(j)}$ for the the $j$-th coordinate of the position $\Xi$ and $q$ is the function given by \eqref{q}  taking $\rho_k=1/2$ for all $k$.
\end{lemma}

We now move on to the  MAP $(\xi^R,\Xi^R)=(\xi^R_t,\Xi^R_t)_{t\geq0}$ underlying the ssMp $R$ from (\ref{item_c}). As we will see in Section~\ref{sec_skorokhod_reflected_stable}, the construction of $R$ is such that reflections on the boundary of the orthant always happen continuously; which forces the modulator $\Xi^R$ to likewise experience continuous reflection as it hits the boundary of $\mathcal{S}^{d,+}_{1}$. That being the case, the process is completely determined by its generator, which describes both the MAPs' jumps and its behaviour near the boundary as it gets reflected.

\begin{definition}
We say that a function $f\in C^2_b(\mathbb{R}\times\mathcal{S}^{d,+}_{1})$ is in \textit{class $\mathcal{D}$} if it satisfies for each $1\leq i\leq d$ and $\boldsymbol{w}\in[0,\infty)^d\setminus\{\boldsymbol{0}_d\}\text{ with } w_i=0$, the condition
	\begin{equation}
		\label{def:8zillion}
	\Big(\boldsymbol{e}_0+\boldsymbol{e}_i-\|\boldsymbol{w}\|^{-1}_1\sum_{j\neq i} w_j\boldsymbol{e}_j\Big)\cdot\gradient f\big(\log{\|\boldsymbol{w}\|_1},\arg(\boldsymbol{w})\big)=0,
\end{equation}
where the set of vectors $\{\boldsymbol{e}_0,\boldsymbol{e}_1,\ldots,\boldsymbol{e}_d\}$ is just the standard orthonormal basis of $\mathbb{R}^{d+1}.$
\end{definition}
We can now state our main result about the MAP underlying the Skorohod reflected stable process. 
\begin{theorem}
	\label{thm_generator_underlying_MAP_skorokhod_reflection}
	Let $(\xi^R,\Xi^R)$ be the underlying MAP of $R$ from (\ref{ssMp_from_item_c_specific}) with respect to the norm $\|\cdot\|_1$. Then, the operator $A:\mathcal{D}\to C^2_b(\mathbb{R}\times\mathcal{S}^{d,+}_{1})$ defined by
\begin{align*}
&Af(x,\boldsymbol{\theta}) \\
&=c_1\Bigg[ \int_0^\infty \Bigg\{\sum_{j=1}^d f\big(x,\boldsymbol{v}_{j,y}^{\boldsymbol{\theta}}\big)-f(x,\boldsymbol{\theta}) \\
&\qquad\qquad-\Big(\boldsymbol{e}_0+(1-\theta_j)\boldsymbol{e}_j-\sum_{k\neq j} \theta_k \boldsymbol{e}_k\Big)\cdot \grad{f}(x,\boldsymbol{\theta}) \mathbbm{1}_{\big(0,\log(1+e^x)-x\big)}(y)\Bigg\}\frac{e^{y}}{(e^{y}-1)^{\alpha}}dy\Bigg], 
\end{align*}
solves the martingale problem for $(\xi^R,\Xi^R)$, where $c_1=\pi^{-1}\Gamma(1+\alpha)\sin(\pi\alpha)$, and the vector $\boldsymbol{v}_{j,y}^{\boldsymbol{\theta}}$ is from (\ref{frequent_vector}).
\end{theorem}

\begin{remark}
Notice that the boundary conditions from the previous result in dimension $d=2$ are precisely:
\[
	 (1,1,-1)^T\cdot \grad{f}(w,0,1)=0\enspace\enspace  \forall w\in\mathbb{R}
	\quad\text{ and }\quad
	 (1,-1,1)^T\cdot \grad{f}(w,1,0)=0\enspace\enspace \forall w\in\mathbb{R}.
\]
\end{remark}

To prove Theorem \ref{thm_generator_underlying_MAP_skorokhod_reflection}, we  will establish a more general result for the ssMp $R$ involving, not one-dimensional stable processes $X^{(i)}$, but more general one-dimensional spectrally-positive L\'evy processes; see Theorem \ref{thm_generator_of_skorokhod_reflected_d_dim_spectrally_positive_levy}. We believe that said result, in its own right, might even be of independent mathematical interest.

The last expository result we give describes, through an SDE, the underlying MAP $(\rho,\Theta)=(\rho_t,\Theta_t)_{t\geq0}$ of the ssMp $\mathcal{R}$ from (\ref{item_d}) in dimension $d=2$. We regard it as a corollary to the result (appearing later in the paper), Theorem~\ref{thm_lp_generator_of_MAP_of_SK_reflected_d_dim_BM}, which reveals the generator of $(\rho,\Theta)$ in the more general setting of dimension $d\geq2$.

\begin{corollary}
	\label{cor_sde_of_underlying_MAP_of_sk_reflected_bm_in_l1_d2}
Assume $d=2.$ Let $(\rho,\Theta)=(\rho_t,\Theta_t)_{t\geq0}$ be the underlying MAP of $\mathcal{R}$ from (\ref{item_d}) with respect to the $L^1$ norm on $\mathbb{R}^2$. Then, $(\rho,\Theta)$ is the unique weak solution to the following SDE:
	\begin{equation*}
		d\begin{pmatrix}
			\rho_t \\ \Theta_t
		\end{pmatrix}=\boldsymbol{a}(\Theta_t) dt + \boldsymbol{b}(\Theta_t) d\boldsymbol{W}_t + \boldsymbol{\gamma}(\Theta_t) d\ell_t,
	\end{equation*}
	where $\boldsymbol{W}$ is $2$-dimensional Brownian motion, $\ell$ is local time of $\Theta$ at the boundary of $\mathcal{S}^{2,+}_{1}$ (i.e., $\{(0,1),(1,0)\}$), and for $\boldsymbol{\theta}=(\theta_1,\theta_2)\in\mathcal{S}^{2,+}_{1}$,
	\begin{equation*}
		\boldsymbol{a}(\boldsymbol{\theta})\coloneqq \begin{pmatrix}
			-1 \\ \theta_1-\theta_2 \\ \theta_2-\theta_1
		\end{pmatrix},\qquad \boldsymbol{\gamma}(\boldsymbol{\theta})\coloneqq\begin{pmatrix}
			1 \\ \mathbbm{1}_{\{(0,1)\}}(\boldsymbol{\theta})-\mathbbm{1}_{\{(1,0)\}}(\boldsymbol{\theta}) \\ \mathbbm{1}_{\{(1,0)\}}(\boldsymbol{\theta})-\mathbbm{1}_{\{(0,1)\}}(\boldsymbol{\theta})
		\end{pmatrix},
	\end{equation*}
	and
	\begin{equation*}
		\boldsymbol{b}(\boldsymbol{\theta})\coloneqq\begin{pmatrix}
			\frac{2\theta_1\theta_2}{\sqrt{\lambda_2(\boldsymbol{\theta})}+\sqrt{\lambda_3(\boldsymbol{\theta})}}+\frac{\sqrt{\lambda_2(\boldsymbol{\theta})}+\sqrt{\lambda_3(\boldsymbol{\theta})}}{2} & -\frac{\theta_1-\theta_2}{\sqrt{\lambda_2(\boldsymbol{\theta})}+\sqrt{\lambda_3(\boldsymbol{\theta})}} \\
			-\frac{\theta_1-\theta_2}{\sqrt{\lambda_2(\boldsymbol{\theta})}+\sqrt{\lambda_3(\boldsymbol{\theta})}} & \frac{\sqrt{\lambda_2(\boldsymbol{\theta})}+\sqrt{\lambda_3(\boldsymbol{\theta})}}{4}-\frac{\theta_1\theta_2}{\sqrt{\lambda_2(\boldsymbol{\theta})}+\sqrt{\lambda_3(\boldsymbol{\theta})}} \\
			\frac{\theta_1-\theta_2}{\sqrt{\lambda_2(\boldsymbol{\theta})}+\sqrt{\lambda_3(\boldsymbol{\theta})}} & \frac{\theta_1\theta_2}{\sqrt{\lambda_2(\boldsymbol{\theta})}+\sqrt{\lambda_3(\boldsymbol{\theta})}} - \frac{\sqrt{\lambda_2(\boldsymbol{\theta})}+\sqrt{\lambda_3(\boldsymbol{\theta})}}{4}
		\end{pmatrix},
	\end{equation*}
	where $\lambda_2(\boldsymbol{\theta})=1+\theta_1^2+\theta_2^2+\sqrt{(1+\theta_1^2+\theta_2^2)^2-2}$ and $\lambda_3(\boldsymbol{\theta})=1+\theta_1^2+\theta_2^2-\sqrt{(1+\theta_1^2+\theta_2^2)^2-2}$.
\end{corollary}

The rest of the paper is organized as follows: in Section~\ref{sec_norm_dep_lamperti_ssMps_MAPs} we provide the full proof of Theorem~\ref{thm_norm_dep_lamperti_transform}. In Section~\ref{sec_stable_killed_exit_cone} we provide the more general dimension $d\geq 2$ treatment and characterization of the underlying MAP of the ssMp $Z$ from (\ref{item_a}). And, in Sections~\ref{sec_reflected_symm_stable}, \ref{sec_skorokhod_reflected_stable} and \ref{sec_reflected_BM} we do the same for the underlying MAPs of the ssMps given in (\ref{item_b}), (\ref{item_c}), and (\ref{item_d}), respectively.

\section{Proof of the norm-dependent Lamperti-type transform between ssMps and MAPs}
\label{sec_norm_dep_lamperti_ssMps_MAPs}

This section concerns the proof of Theorem~\ref{thm_norm_dep_lamperti_transform}. We recall that $E$ is assumed to be a locally compact and separable Banach space over the field of reals. We let $\mathcal{E}$ be a sigma-algebra on $E$, and $\partial$ a point not in $E$ which acts as a cemetery state. We define $(E_{\partial},\mathcal{E}_\partial)\coloneqq(E\cup\{\partial\},\sigma(E\cup\{\partial\}))$ to be the measurable space $(E,\mathcal{E})$ \textit{augmented by $\partial$}; cf., \cite{blumenthal2007markov}. A standard practice (which we implicitly employ throughout the paper actually) is to extend an $\mathcal{E}$-measurable function $f:E\to\mathbb{R}$ to an $\mathcal{E}_\partial$-measurable function by setting $f(\partial)=0$.

\begin{proof}[Proof of Theorem~\ref{thm_norm_dep_lamperti_transform}]
	We begin by proving the first part of the theorem. To this end, we denote the law $(\xi,\Xi)$ by $\left(\mathbb{P}_{(\rho,\theta)}, (\rho,\theta)\in\mathbb{R}\times\mathcal{S}_{\|\cdot\|}\right)$. Denote by $Z=(Z_t)_{t\geq0}$ the process with the $\|\cdot\|$-polar decomposition from (\ref{thm_norm_rep_process}). Observe that the law of $Z$ issued from $Z_0=z_0\in E,$ here denoted by $\mathrm{P}_{z_0}$, is the image measure of  $\mathbb{P}_{(\|z_0\|, \arg(z_0))}$ under the transformation in \eqref{thm_norm_rep_process}. The assertions regarding the cemetery state and lifetime of $Z$ follow from the fact that the function $s\mapsto\int_0^s {\rm e}^{\alpha\xi_r}dr$ is non-decreasing and continuous, and has a terminal value $\int^{\zeta}_0\exp\{\alpha\xi_s\}ds.$ Hence, $K=\int^{\zeta}_0\exp\{\alpha\xi_s\}ds$ is the lifetime of $Z$, as defined in (\ref{thm_norm_rep_process}). That the paths of $Z$ are c\`adl\`ag follows immediately from the fact $(\xi,\Xi)$ has c\`adl\`ag paths and the fact that the time-change $\phi$ is right-continuous. The Markov property of $Z$ follows from that of the MAP $(\xi,\Xi)$ using Proposition~7.9 in \cite{kallenberg1997foundations}.  Proving the self-similarity property of $Z$,
	\begin{equation*}
		((Z_t)_{t\geq0},P_{cx})\stackrel{\text{d}}{=}((cZ_{c^{-\alpha}t})_{t\geq0}, P_x),\qquad \forall c>0, x\in E;
	\end{equation*}
	is equivalent to proving that for every $c>0$ and $x\in E,$ we have
	\begin{equation}
		\label{ss_wts}
		(\{(\exp\{\xi_{\phi(t)}\},\Xi_{\phi(t)}),t\geq0\},\mathbb{P}_{(\eta',\theta)})\stackrel{\text{d}}{=}(\{(c\exp\{\xi_{\phi(c^{-\alpha}t)}\},\Xi_{\phi(c^{-\alpha}t)}),t\geq0\},\mathbb{P}_{(\eta,\theta)}),
	\end{equation}
	where $\eta'=\log(\|x\|)+\log(c)$ and $(\eta,\theta)=(\log(\|x\|),\frac{1}{\|x\|}x).$ The latter is an easy consequence of the translation invariance property given in \eqref{additive_lemma}.
	
We now deal with the second part of the theorem. For that end we will prove that the $\mathbb{R}\times\mathcal{S}_{\|\cdot\|}$-valued process $(\xi,\Xi)$, defined by the transformation
	\begin{equation}
		%\label{def_of_underlying_norm_dep_MAP}
		(\xi_t,\Xi_t)=\begin{cases}
			(\log(\|Z_{I_t}\|),\frac{1}{\|Z_{I_t}\|}Z_{I_t}), & \text{if $t<\zeta$}\\
			(-\infty,\partial), & \text{if $t\geq \zeta$};
		\end{cases}
		\label{MAPgendef}
	\end{equation}
	where $I_t\coloneqq\inf\{s>0:\int_0^s \|Z_r\|^{-\alpha} dr>t\}$, $t\geq 0;$  $\zeta:=\int^{K}_0 \|Z_r\|^{-\alpha} dr,$ $K:=\inf\{r>0: \|Z_r\|=0\};$ 
	is a MAP with cemetery state $(-\infty,\partial)$ and lifetime $\zeta$. The law of $(\xi,\Xi),$ denoted by $\left(\mathbb{P}_{x,\theta}, (x,\theta)\in\mathbb{R}\times\mathcal{S}_{\|\cdot\|}\right),$ is the image measure of $\left(P_{z_0}, z_0\in E\right),$ under the above transformation.  Indeed, that this process has c\'adl\'ag paths is a consequence of the fact that $Z$ does and that the time change involved in the construction of $(\xi,\Xi),$ is right-continuous. The couple $(\xi,\Xi)$ also inherits the strong Markov property from that of $Z,$ see e.g. Section III.21 in \cite{rogers2000diffusions}. To verify that the property \eqref{MAP_eqn2} holds, we proceed as follows. Let $g:\mathbb{R}\times E\to\mathbb{R}$  be any bounded measurable function $s,t\geq0$ and $(x,\theta)\in E\times \mathcal{S}_{\|\cdot\|}$. On $\{t<\zeta\}$ we have
	\begin{align*}
		&\mathbb{E}_{(x,\theta)}[g(\xi_{t+s}-\xi_t,\Xi_{t+s})\mathbbm{1}_{\{t+s<\zeta\}}|\mathcal{G}_t]\notag\\
		&=E_{{\rm e}^x\theta}\Bigg[g\Bigg(\log(\frac{\|Z_{I_t+I_s\circ\theta_{I_{t}}}\|}{\|Z_{I_0\circ\theta_{I_{t}}}\|}),\frac{1}{\|Z_{I_t+I_s\circ\theta_{I_{t}}}\|}Z_{I_t+I_s\circ\theta_{I_{t}}}\Bigg)\mathbbm{1}_{\{I_t+I_s\circ\theta_{I_{t}}<K\}}\Bigg|\mathcal{F}_{I_t}\Bigg] \\
		&=E_{{\rm e}^x\theta}\Bigg[ g\Bigg(\log(\frac{\|Z_{I_t+I_s\circ\theta_{I_{t}}}\|}{\|Z_{I_0\circ\theta_{I_{t}}}\|}),\frac{1}{\|Z_{I_t+I_s\circ\theta_{I_{t}}}\|}Z_{I_t+I_s\circ\theta_{I_{t}}}\Bigg)\mathbbm{1}_{\{I_s\circ\theta_{I_{t}}<K\circ\theta_{I_{t}}\}}\Bigg|\mathcal{F}_{I_t}\Bigg]\\
		&=E_{{\rm e}^x\theta}\Bigg[\Bigg(g\Bigg(\log(\frac{\|Z_{I_s}\|}{\|Z_{0}\|}),\frac{1}{\|Z_{I_s}\|}Z_{I_s}\Bigg)\mathbbm{1}_{\{I_s<K\}}\Bigg)\circ\theta_{I_{t}}\Bigg|\mathcal{F}_{I_t}\Bigg]\\
		&=E_{Z_{I_t}}\Bigg[g\Bigg(\log(\frac{\|Z_{I_s}\|}{\|Z_{0}\|}),\frac{1}{\|Z_{I_s}\|}Z_{I_s}\Bigg)\mathbbm{1}_{\{I_s<K\}}\Bigg],
	\end{align*}
	where the last line is an immediate consequence of the strong Markov property of $Z$ and that $I_t$ is a stopping time for it. To conclude the proof, we rewrite the last expression as 
	\begin{equation*}
	\begin{split}
	&E_{Z_{I_t}}\Bigg[g\Bigg(\log(\frac{\|Z_{I_s}\|}{\|Z_{0}\|}),\frac{1}{\|Z_{I_s}\|}Z_{I_s}\Bigg)\mathbbm{1}_{\{I_s<K\}}\Bigg]\\
	&=\mathbb{E}_{\log\|Z_{I_t}\|,\arg(Z_{I_t})}\left[g(\xi_s-\xi_{0}, \Xi_s)\mathbbm{1}_{\{s<\zeta\}}\right]\\
	&=\mathbb{E}_{0,\arg(Z_{I_t})}\left[g(\xi_s, \Xi_s)\mathbbm{1}_{\{s<\zeta\}}\right]. \qed
	\end{split} 
	\end{equation*}

\end{proof}

\section{Stable process killed upon exiting the orthant}
\label{sec_stable_killed_exit_cone}

In this section we consider the high-dimensional analogue of the pssMp obtained by killing a stable L\'evy process upon its first exit from the positive half-line studied in \cite{caballero2006conditioned}. First, let $X=(X_t)_{t\geq0}$ be the $d$-dimensional $\alpha$-stable process with index of self-similarity $\alpha\in(0,2)$ given by
\begin{equation}
	\label{def_of_our_d_dim_stable}
	X_t=(X^{(1)}_t,\ldots,X^{(d)}_t)\in\mathbb{R}^d,\qquad t\geq 0,
\end{equation}
where $X^{(1)}=(X_t^{(1)})_{t\geq0},\ldots,X^{(d)}=(X_t^{(d)})_{t\geq0}$ is a collection of independent (one-dimensional) $\alpha$-stable processes with positivity parameters $\rho_1,\ldots,\rho_d$, respectively, that admit both positive and negative jumps. Let us denote the probabilities of the ssMp $X$ of this section by $(P_{\boldsymbol{x}}, \boldsymbol{x}\in\mathbb{R}^d)$, and the expectation with respect to $P_{\boldsymbol{x}}$ by $E_{\boldsymbol{x}}$.  It is clear that $X$ is an ssMp. It follows that the process $Z=(Z_t)_{t\geq0}$ defined by 
\begin{equation}
	\label{ssMp_a_specific}
Z_t=X_t\mathbbm{1}_{\{t<\tau^D\}},\qquad t\geq0,
\end{equation}
where $\tau^D\coloneqq \inf\{t>0:X_t\notin (0,\infty)^d\}$ is an example of an ssMp absorbed at the zero vector, and thus admits an underlying MAP with respect to the $L_1$-norm, $\norm{\cdot}_1$. Let us also denote the probabilities of $(\xi,\Xi)$ by $\mathbb{P}_{(x,\boldsymbol{\theta})}$, $(x,\boldsymbol{\theta})\in\mathbb{R}\times\mathcal{S}_{1}^{d,+}$; which we know are obtained from those of $Z$ killed at its first hitting time of $0$ through the description given in the second part of Theorem \ref{thm_norm_dep_lamperti_transform}. We shall denote the expectation with respect to $\mathbb{P}_{(x,\boldsymbol{\theta})}$ by $\mathbb{E}_{(x,\boldsymbol{\theta})}$.

To be consistent with the introductory section, we denote this MAP by $(\xi,\Xi)=(\xi_t,\Xi_t)_{t\geq0}$, and denote its lifetime by $\zeta$.

Since each $X^{(i)}$ admits both positive and negative jumps, standard theory says that $X^{(i)}$ does not creep downwards (see Exercise 7.4 (i) of \cite{kyprianou2014fluctuations}). This is to say that each $X^{(i)}$ must exit $(0,\infty)$ via a jump, almost-surely. From the fact that the processes $X^{(i)}, i\in\{1,\ldots, d\}$, are independent, and hence do not have simultaneous jumps, it is straightforward to derive from e.g., \cite{ken1999levy}, that the L\'evy measure of $X$ from (\ref{def_of_our_d_dim_stable}),  which we denote by $\Pi$, satisfies
	\begin{equation}
		\label{levy_measure_of_d_dim_stable}
		\Pi(d\boldsymbol{x})=\sum_{j=1}^d (\bigotimes_{k\neq j} \delta_0(dx_k))\otimes \Pi^{(j)}(dx_j),\qquad\boldsymbol{x}=(x_1,\ldots,x_d)\in\mathbb{R}^d,
	\end{equation}
	 where $\Pi^{(j)}$ is the L\'evy measure of $X^{(j)}$, necessarily taking the form  given in \eqref{def_levy_measures_of_1dim_indep_stables}.

We may also remark in this respect that for every $\lambda>0$ and $\bm x\in\mathbb{R}^d\setminus\{\boldsymbol{0}_{d}\}$, 
	\begin{equation}
	\Pi(\lambda d\boldsymbol{x})=\lambda^{-\alpha} \Pi(d\boldsymbol{x}).\label{lemma_levy_measure_pullout_property}
	\end{equation}

We start off by working out the killing rate function of $(\xi,\Xi)$, that is, we start off by proving Lemma~\ref{compact_thm_killing_rate}.
\begin{proof}[Proof of Lemma~\ref{compact_thm_killing_rate}]
Going back to the proof of Theorem~\ref{thm_norm_dep_lamperti_transform}, we see from (\ref{def_of_underlying_norm_dep_MAP}) that the killing time, $\zeta$, of $(\xi,\Xi)$ equals $$\zeta=\int_0^{\tau^D} \|X_s\|_1^{-\alpha} ds,\qquad \tau^D\coloneqq\inf\{s>0:X_s\notin(0,\infty)^d\}.$$ Fix $\boldsymbol{\theta}=(\theta_1,\ldots,\theta^d)\in\mathcal{S}_{1}^{d,+}$ and a positive bounded measurable function $f$ on $[0,\infty)$.  Now, recalling that stable processes in this subsection don't creep, and hence, that $\tau^D$ is a jump time of $X$, we get
\begin{equation*}
\mathbb{E}_{(0,\boldsymbol{\theta})}[f(\zeta)]=E_{\boldsymbol{\theta}}\left[\sum_{t\geq0}f\left(\int_0^{t} \|X_s\|_1^{-\alpha} ds\right)\mathbbm{1}_{\{X_{t-}\in(0,\infty)^d,X_t\notin(0,\infty)^d\}}\mathbbm{1}_{\{t<\tau^D\}}\right].
\end{equation*} 
 Thanks to the compensation formula, and property \eqref{lemma_levy_measure_pullout_property}, the latter expression takes the form
\begin{align*}
\mathbb{E}_{(0,\boldsymbol{\theta})}[f(\zeta)]&= E_{\boldsymbol{\theta}}\left[\int_0^{\tau^D}\int_{ \mathcal{M}_t}f\left(\int_0^{t} \|X_s\|_1^{-\alpha} ds\right)\Pi(d\boldsymbol{z})dt\right],
\end{align*}
where,
\begin{equation*}
\mathcal{M}_t\coloneqq \bigcup_{j=1}^d \Big( \{0\}^{j-1}\times (-\infty,-X^{(j)}_t)\times\{0\}^{d-j}\Big),\qquad t<\tau^D.
\end{equation*}
By then performing the change-of-variables $\boldsymbol{w}=\|X_t\|_1^{-1}\boldsymbol{z}$, and denoting the the $k$-th coordinate of the vector $\arg(X_t)$ by $\arg(X_t)^{(k)}$, one obtains
\begin{equation*}
\begin{split}
\mathbb{E}_{(0,\boldsymbol{\theta})}[f(\zeta)]&= E_{\boldsymbol{\theta}}\Bigg[\int_0^{\tau^D}\int_{ \mathcal{M}'_t}f\left(\int_0^{t} \|X_s\|_1^{-\alpha} ds\right)\Pi\big(\|X_t\|_1d\boldsymbol{w}\big)dt\Bigg]\notag\\
&=E_{\boldsymbol{\theta}}\Bigg[\int_0^{\tau^D}f\left(\int_0^{t} \|X_s\|_1^{-\alpha} ds\right)\Big(\int_{ \mathcal{M}'_t}\Pi(d\boldsymbol{w})\Big)\|X_t\|_1^{-\alpha}dt\Bigg],
\end{split}
\end{equation*}
where $\mathcal{M}'_t$ is simply the set $\mathcal{M}_t$ with $X^{(j)}_t$ therein replaced by $\arg(X_t)^{(j)}$ for every $j$, and we have applied \eqref{lemma_levy_measure_pullout_property} in the final equality. By performing the change of variables
\begin{equation*}
t=I_v\coloneqq\inf\{s>0:\int_0^s \|X_r\|_1^{-\alpha} dr>v\},\qquad dv=\|X_t\|_1^{-\alpha}dt,
\end{equation*}
and recalling the definition of $(\xi,\Xi)$, we get
\begin{align*}
\mathbb{E}_{(0,\boldsymbol{\theta})}[f(\zeta)]&=\mathbb{E}_{(0,\boldsymbol{\theta})}\Bigg[\int_0^{\zeta}f(v)\Big(\sum_{j=1}^d \Pi^{(j)}\big((-\infty,-\Xi^{(j)}_t)\big)\Big)dv\Bigg] \\
&=\frac{1}{\alpha}\sum_{j=1}^d c_2^{(j)} \mathbb{E}_{(0,\boldsymbol{\theta})}\Bigg[\int_0^{\zeta} \big(\Xi_v^{(j)}\big)^{-\alpha}f(v)dv\Bigg].
\end{align*} Defining the function $q:\mathcal{S}^{d,+}_{1}\to[0,\infty)$,
\begin{equation*}
q(\boldsymbol{\theta})=\frac{1}{\alpha}\sum_{j=1}^d c_2^{(j)}\theta^{-\alpha}_j,\qquad \boldsymbol{\theta}=(\theta_1,\ldots,\theta_d)\in\mathcal{S}^{d,+}_{1}; 
\end{equation*} {\color{black} where the $c^{(j)}_2$ are the respective constants in the jump measure of $X^{(j)}$ as per \eqref{c-s} with $\rho = \rho_j$.}
The previous identity can be written as 
\begin{align*}
\mathbb{E}_{(0,\boldsymbol{\theta})}[f(\zeta)]&= \mathbb{E}_{(0,\boldsymbol{\theta})}\left[\int_0^{\zeta} q(\Xi_v)f(v)dv\right],
\end{align*}
By setting $f(v)=\mathbbm{1}_{(0,T)}(v)$, $T>0$, above and computing the derivative with respect to $T$ evaluated at $0$ one obtains the killing rate. \qed
\end{proof}

\begin{proof}[Proof of Theorem~\ref{thm_quintuple_law_up_to_first_exit_l1_norm}]
From the construction of $(\xi,\Xi)$ (recall \eqref{def_of_underlying_norm_dep_MAP}), we have, for any $t>0$, any positive bounded measurable function $f$ on $\mathbb{R}^{2}\times \mathcal{S}^{d,+}_{1}\times \mathcal{S}^{d,+}_{1},$
\begin{align*}
&\mathbb{E}_{(0,\boldsymbol{\theta})}\Big[\sum_{s\leq t} f\big(\xi_{s-},\Delta \xi_s,\Xi_{s-},\Xi_s\big)\mathbbm{1}_{\{\Delta\xi_s\neq 0\}}\Big] \\
&=E_{\boldsymbol{\theta}}\Bigg[\sum_{s\leq t} f\Big(\log{\|X_{I_s-}\|_1},\log{\frac{\|X_{I_s}\|_1}{\|X_{I_s-}\|_1}},\arg(X_{I_s-}),\arg(X_{I_s})\Big)\mathbbm{1}_{\{\Delta X_{I_s}\neq0\}}\Bigg] \\
&=E_{\boldsymbol{\theta}}\Bigg[\sum_{s\leq I_t} f\Big(\log{\|X_{s-}\|_1},\log{\frac{\|X_{s}\|_1}{\|X_{s-}\|_1}},\arg(X_{s-}),\arg(X_{s})\Big)\mathbbm{1}_{\{\Delta X_{s}\neq0\}}\Bigg].
\end{align*}
Now, since for any $t>0,$  $I_t$ is a stopping time, applying the compensation formula for the Poisson random measure of the jumps of $X$, we have that the rightmost term in the above equation equals:
\begin{align*}
&E_{\boldsymbol{\theta}}\left[\sum_{s\leq I_t} f\Big(\log{\|X_{s-}\|_1},\log{\frac{\|X_{s}\|_1}{\|X_{s-}\|_1}},\arg(X_{s-}),\arg(X_{s})\Big)\mathbbm{1}_{\{\Delta X_{s}\neq0\}}\right]\\
&=E_{\boldsymbol{\theta}}\left[\int^{I_t}_{0}ds \int\Pi(d\boldsymbol{x})\mathbbm{1}_{\{X_{s}+\boldsymbol{x}\in D\}}f\Big(\log{\|X_{s}\|_1},\log{\frac{\|X_{s}+\boldsymbol{x}\|_1}{\|X_{s}\|_1}},\arg(X_{s}),\arg(X_{s}+\boldsymbol{x})\Big)\right].
\end{align*}
Now, making a change of variables $s=I_v,$ and then using property\eqref{lemma_levy_measure_pullout_property}, the latter expression can be written as 
\begin{align*}
&=E_{\boldsymbol{\theta}}\left[\int^{t}_{0}dv \|X_{I_v}\|_1^{\alpha}\int\Pi(d\boldsymbol{x})\mathbbm{1}_{\{X_{I_v}+\boldsymbol{x}\in D\}} f\Big(\log{\|X_{I_v}\|_1},\log{\frac{\|X_{I_v}+\boldsymbol{x}\|_1}{\|X_{I_v}\|_1}},\arg(X_{I_v}),\arg(X_{I_v}+\boldsymbol{x})\Big)\right]\\
&=E_{\boldsymbol{\theta}}\left[\int^{t}_{0}dv \int\Pi(d\boldsymbol{x})\mathbbm{1}_{\{X_{I_v}+ \|X_{I_v}\|_1\boldsymbol{x}\in D\}}\right.\\
&\hspace{3cm}\left. \times  f\Big(\log{\|X_{I_v}\|_1},\log{\frac{\|X_{I_v}+ \|X_{I_v}\|_1\boldsymbol{x}\|_1}{\|X_{I_v}\|_1}},\arg(X_{I_v}),\arg(X_{I_v}+ \|X_{I_v}\|_1\boldsymbol{x})\Big)\right].
\end{align*}
Appealing again to \eqref{def_of_underlying_norm_dep_MAP}, the latter can further be written as
\begin{align*}
&\mathbb{E}_{(0,\boldsymbol{\theta})}\Big[\sum_{s\leq t} f(\xi_{s-},\Delta \xi_s,\Xi_{s-},\Xi_s)\mathbbm{1}_{\{\Delta\xi_s\neq 0\}}\Big] \\
&=\mathbb{E}_{(0,\boldsymbol{\theta})}\Bigg[\int_0^{t\wedge K} \int_{\mathbb{R}} \Bigg\{\sum_{j=1}^d \Big(c_2^{(j)}\mathbbm{1}_{(-\Xi_v^{(j)},0)}(w)+c_1^{(j)}\mathbbm{1}_{(0,\infty)}(w)\Big)\cdot \\
&\hspace{3cm}\times f\Big(\xi_v,\log{\|\Xi_v+w\boldsymbol{e}_j\|_1},\Xi_v,\arg(\Xi_v+w\boldsymbol{e}_j)\Big)|w|^{-(1+\alpha)}\Bigg\}dwdv\Bigg].
\end{align*}
To finish, we still need to make a change of variables. For that end, we note that $v<\zeta$ and $w>-\Xi^{(j)}_v$ imply $\Xi_v\in (0,\infty)^d$ and $\Xi_v^{(j)}+w>0$; in particular, $|\Xi^{(k)}_v|=\Xi^{(k)}_v$, for $k\neq j$, and $|\Xi_v^{(j)}+w|=w+\Xi_v^{(j)}$. As such, for $v<\zeta$ and $w>-\Xi^{(j)}_v$,
\begin{equation*}
	\|\Xi_v+w\boldsymbol{e}_j\|_1=\Big( w+\Xi_v^{(j)}+\sum_{k\neq j} \Xi_v^{(k)} \Big)=\Big( w+\Xi_v^{(j)}+1-\Xi_v^{(j)} \Big)=1+w,
\end{equation*}
where we have used the fact that $\Xi_v\in\mathcal{S}^{d,+}_{1}$ in the last two equalities. The result follows by plugging this into the above formula for the expectation, then performing the following change-of-variables:
\begin{equation*}
	y= \log(w+1),\enspace w={\rm e}^{y}-1,\enspace \frac{dw}{dy}={\rm e}^{y},
\end{equation*}
and lastly, bringing the resulting expression in the desired form displayed in the RHS of \eqref{def_levy_systems_for_MAPs_ver3} in terms of an appropriate transition kernel $L$. One will then indeed find that the said transition kernel is precisely that which is given in the statement of the theorem. \qed
\end{proof}

\section{Reflected symmetric stable process}
\label{sec_reflected_symm_stable}

In this section we consider a high-dimensional analogue of the pssMp obtained by reflecting the path of a symmetric stable L\'evy process in the $x$-axis, which, together with its underlying L\'evy process through the Lamperti transform, were studied in \cite{kyprianou2014hitting}. More precisely, we study the $d$-dimensional ssMp $\hat{X}=(\hat{X}_t)_{t\geq0}$ defined by
\begin{equation}
	\label{def_reflected_symm_d_dim_stable}
\hat{X}_t=(|X^{(1)}_t|,\ldots,|X^{(d)}_t|),\qquad t\geq0,
\end{equation} 
where the $X^{(i)}$ are the $1$-dimensional stable processes from (\ref{def_of_our_d_dim_stable}) with $\alpha\in(0,2)$ and $\rho_i=\frac{1}{2}$, and $|\cdot|$ denotes the usual absolute value in $\mathbb{R}$. We assume this process dies and gets absorbed at the first time it hits the zero vector $\boldsymbol{0}_d\in\mathbb{R}^{d}$.

Like in the one-dimensional version of this particular example, it would be natural to also expect  the underlying MAP $(\hat{\xi},\hat{\Xi})=(\hat{\xi}_t,\hat{\Xi}_t)_{t\geq0}$ of $\hat{X}$ to be interconnected with $(\xi,\Xi)$, the underlying MAP of the process $Z=(Z_t)_{t\geq0}$ from last section. Indeed, it is clear that up to, but not including, the killing time, $\zeta$, of $(\xi,\Xi)$, the paths of these two MAPs coincide, but then at time $\zeta$, instead of being sent to the cemetery state, $(\hat{\xi},\hat{\Xi})$ has to make a ``corrective" jump to stay in the orthant. By the strong Markov property, this procedure is then repeated ad infinitum. Once again, we refer interested readers to the paper \cite{kyprianou2014hitting}, as well as Chapters 5.6 and 5.7 of \cite{kyprianou2022stable}, for the one-dimensional analogue of this construction.

We first introduce some notation: since the $X^{(i)}$ of this section are chosen to have the same positivity parameters $\rho_i= \frac{1}{2}$, it follows that their respective L\'evy measures all agree and are equal to the measure $\Pi$ from (\ref{def_levy_measures_of_1dim_indep_stables}) with 
\begin{equation}
	\label{def_of_special_const_c}
c_1=c_2=c\coloneqq \pi^{-1}\Gamma(1+\alpha)\sin(\pi\alpha/2).
\end{equation}

To proceed with our analysis, we need to define the following \textit{"negation" operator} on $\mathbb{R}^d$:
\begin{definition}
	For $j\in\{1,\ldots,d\}$, define the operator $N^{(j)}:\mathbb{R}^d\to\mathbb{R}^d$ as follows:
	\begin{align*}
		N^{(j)}(\boldsymbol{x})&=(x_1,x_2,\ldots,x_{j-1},-x_j,x_{j+1},\ldots,x_d),\qquad \boldsymbol{x}=(x_1,\ldots,x_d)\in\mathbb{R}^d \\
		&=-x_j\boldsymbol{e}_j+\sum_{i\neq j} x_i \boldsymbol{e}_i,
	\end{align*}	
where as usual, $\{\boldsymbol{e}_j:1\leq j\leq d\}$ denotes the standard orthonormal basis of $\mathbb{R}^d$. The \textit{reflection operator} $R:\mathbb{R}^d\to\mathbb{R}^d$ is defined by
\begin{equation*}
	R(\boldsymbol{x})=\sum_{j=1}^d \mathbbm{1}_{\{x_j<0\}}(\boldsymbol{x}) N^{(j)}(\boldsymbol{x}),\qquad \boldsymbol{x}=(x_1,\ldots,x_d)\in\mathbb{R}^d.
\end{equation*}	
\end{definition}	

We will be applying this negation operator solely to the process $X=(X^{(1)},\ldots,X^{(d)})$ at the times when it jumps out of the first orthant. We know that such a jump can only occur in exactly one of $d$ possible directions, which is to say that at each of $X$'s exit times from the orthant, one and only one of its coordinates $X^{(i)}$ becomes negative. As such, what the above-defined reflection operator $R$ does to $X$ at each of these jump times is to identify which of the $d$ coordinates has become negative and then multiply that coordinate by $-1$ and make it positive/reflect it.

The following recursive construction of $\hat{X}$ will be useful. For $k\in\mathbb{N}$, define the sequence of processes $\hat{X}^{(k)}=(\hat{X}^{(k)}_t)_{t\geq0}$ by 
\begin{equation*}
	\hat{X}^{(k)}_t = 
	\begin{cases}
		\hat{X}^{(k-1)}_t, &\text{if $t<\tau^D_{k-1}$} \\
		R(\hat{X}^{(k-1)}_{\tau^D_{k-1}})	+ \tilde{X}^{(k)}_t, &\text{if $t\geq \tau^D_{k-1}$}
	\end{cases},
\end{equation*}
where $(\hat{X}^0_t)_{t\geq0}=(X_t)_{t\geq0}$, and $\tau^D_0=\tau^D$, $\tilde{X}^{(k)}_t\coloneqq \hat{X}^{(k-1)}_t-\hat{X}^{(k-1)}_{\tau^D_{k-1}}$ (for $t\geq \tau^D_{k-1}$), $\tau^D_{k}\coloneqq\inf\{{\color{black}t>\tau^D_{k-1}}:\hat{X}^{(k)}_t\notin D\}$.  Then for all $k\in\mathbb{N}$, $\hat{X}$ is compatible  with $\hat{X}^{(k)}$ on the time interval $0\leq t<\tau^D_{k}$.

{\color{black} Because $\hat{X}$ is confined to the positive orthant and is a self-similar Markov process, we can identify its MAP $(\hat{\xi},\hat{\Xi})$ simply by the general transformation \eqref{MAPgendef}. }

%{\color{blue}On the other hand, we can also define compatible recursive definition that aligns with the one above for $\hat{X}$. To this end, let
% $(\xi^*,\Xi^*)$ be a copy of the MAP $(\xi,\Xi)$ with the effects of killing removed. From this we obtain the (recursive) definition of $\hat{\Xi}$ in terms of $\Xi^*$:
%\begin{equation}
%	\label{recursive_def_of_reflected_modulator}
%	\hat{\Xi}^{(k)}_t = 
%	\begin{cases}
%		\hat{\Xi}^{(k-1)}_t, &\text{if $t<{\color{red}\hat{\phi}}(\tau^D_{k-1})$} \\
%		R(\hat{\Xi}^{(k-1)}_{\tilde{\phi}(\tau^D_{k-1})})	+ \tilde{\Xi}^{(k)}_t, &\text{if $t\geq {\color{red}\hat{\phi}}(\tau^D_{k-1})$}
%	\end{cases},
%\end{equation}
%where $(\hat{\Xi}^0_t)_{t\geq0}=(\Xi^*_t)_{t\geq0}$,  $\tilde{\Xi}^{(k)}_t\coloneqq {\color{red}\hat{\Xi}^{(k-1)}_t - \hat{\Xi}^{(k-1)}_{\hat{\phi}(\tau^D_{k-1})}}$ (for {\color{red} $t\geq\hat{\phi}(\tau^D_{k-1})$}) {\color{red} and
%\[
%\hat{\phi}(\tau^D_{k-1}) = \int_0^{\tau^D_{k-1}} \norm{\hat{X}_s}_1^{-\alpha}ds. 
%\]} 
%We arrive at the following decomposition of the MAP $(\hat{\xi},\hat{\Xi})$ in terms of $(\xi^*,\Xi^*)$:
%}

{\color{black}
\begin{theorem}
	\label{thm_reflected_MAP_characterization}
	%The underlying MAP, $(\hat{\xi},\hat{\Xi})$, of $\hat{X}$ from (\ref{def_reflected_symm_d_dim_stable}) with respect to norm $\|\cdot\|_1$ on $\mathbb{R}^d$ can be decomposed as follows
	%\begin{enumerate}[(i)]
	%	\item 	
	%The MAP $(\hat{\xi}, \hat{\Xi})$ is defined recursively as follows. 
	Let
 $(\xi^*,\Xi^*)$ be a copy of the MAP $(\xi,\Xi)$ with the effect of killing removed. Moreover, let $L^*\geq0$ be a random variable dependent on $\Xi^*$, whose conditional survival probability satisfies $\Pr(L^*>t \,|\, \Xi^*) = \exp(-\int_0^t q(\Xi^*_s)ds)$, where $q$ is given by \eqref{q} with  $\rho_1 = \dots = \rho_d=\frac{1}{2}$. 
 \begin{enumerate}[(ii)]
 \item[(i)] Recalling $L$ is the time of the  first corrective jump of $(\hat{\xi}, \hat{\Xi})$, then, for each $x\in\mathbb{R}$ and $\boldsymbol{\theta}\in\mathcal{S}^{d,+}_1$, under $\mathbb{P}_{x,\boldsymbol{\theta}}$, $(\hat{\xi}_t, \hat{\Xi}_t)_{0\leq t<L}$ is equal in law to $(\xi^*, \Xi^*)_{0\leq t<L^*}$ with $(\xi^*_0, \Xi^*_0)=(x, \boldsymbol{\theta})$. Moreover, the joint law of $(\Delta \hat{\xi}_L, \hat{\Xi}_{L})$ is given by Lemma \ref{thm_concise_ordinate_corrective_jump}. Given $(\hat{\xi}_L, \hat{\Xi}_L)$, thanks to the Strong Markov property, the 
process $(\hat{\xi}, \hat{\Xi})$ can be constructed iteratively until the next corrective jump  and so on. 

\item[(ii)]
As such, conditional on the sequence of pairs $(L_n, \hat{\Xi}_{L_n-})$, $n= 0,1,2,\dots$, with $L_n$, $n=0,1,2,\dots$ ($L_0 = 0$ and $L_1 = L$) are the corrective jump times of $\hat{\Xi}$, we have 
\begin{equation*} 
\hat{\xi}_t \stackrel{\text{d}}{=}  \xi^*_t+\sum_{k=0}^{N_t} F_k,
 \end{equation*}
		where, $N_t = \sup\{k\in\mathbb{N}: L_k\leq t\}$, and $F_1,F_2,\ldots$ are independent  such that, for Borel $A$ in $\mathbb{R}$, $\Pr(F_k \in A\,|\, \hat{\Xi}_{L_k-})$  is given by \eqref{eq:234zillion} with $G(x, \boldsymbol{\theta}) = \mathbf{1}_{(x\in A)}$.
%		
%		$\stackrel{\text{d}}{=}\Delta \hat{\xi}_{K}$, $\hat{\Xi}$ is the process defined recursively in terms of $\Xi^*$ in (\ref{recursive_def_of_reflected_modulator}) and $(N_t(\hat{\Xi}_t))_{t\geq0}$ ia a non-decreasing $\mathbb{N}_0$-valued process starting from $0$ which makes only jumps of size $1$, and does so at rate $q(\hat{\Xi}_t) dt$, where $q$ is the killing rate function from Lemma~\ref{compact_thm_killing_rate} with $\rho_k=\frac{1}{2}$.
%		%\item 
%		
%		The jump structure of the ``corrective" jumps of the MAP $(\hat{\xi},\hat{\Xi})$ is fully characterized by Lemma \ref{thm_concise_ordinate_corrective_jump}. 
%	
\end{enumerate}
\end{theorem}
}
\begin{proof}
	The first assertion of part (i) follows immediately from the Markov property of $(\hat{\xi},\hat{\Xi})$, agreeing dynamics of $(\xi,\Xi)$ and $(\hat{\xi},\hat{\Xi})$ for times up to (but not including) time $\zeta$ (the killing time of $(\xi,\Xi)$), and the fact that the time $\zeta$ for $(\xi,\Xi)$, which occurs at rate $q(\hat{\Xi}_t) dt$, corresponds to  the first reflection effect that $(\hat{\xi},\hat{\Xi})$ experiences.
	
	We now prove the second assertion of part (i). As alluded to above, thanks to the Markov property of $(\hat{\xi},\hat{\Xi})$ it suffices to study only the distribution of $((\hat{\xi}_t,\hat{\Xi}_t):t\leq L)$ and to iterate. %Due to agreeing dynamics of $(\hat{\xi},\hat{\Xi})$ and $(\xi,\Xi)$ for times up to (but not including) time $L$, there is nothing unknown about the distribution of $(\hat{\xi},\hat{\Xi})$ for such times. 
	It remains to study
	\begin{equation} \label{imp_char_eq} (\hat{\xi}_{L},\hat{\Xi}_{L})=(\hat{\xi}_{L-}+\Delta \hat{\xi}_{L},\hat{\Xi}_{L}).%=(\xi_{L-}+\Delta \hat{\xi}_{L},\hat{\Xi}_{L}).
	\end{equation}
	The RHS of (\ref{imp_char_eq}) clearly indicates that studying $(\hat{\xi}_{L},\hat{\Xi}_{L})$ is equivalent to studying $(\Delta \hat{\xi}_{L},\hat{\Xi}_{L})$ conditionally on $\hat{\Xi}_{L-}$, since $\hat{\Xi}_{L}$ is  dependent on $\hat{\Xi}_{L-}$ and the fact that $(\hat{\xi}, \hat{\Xi})$ is a MAP means that $\Delta \hat{\xi}_{L}$ is conditionally independent of $\hat{\xi}_{L-}$. 
	
	Recall  $H:\mathcal{S}^{d,+}_{1}\to\mathbb{R}$ is the function
	\begin{equation*}
		H(\boldsymbol{\phi})=q(\boldsymbol{\phi})^{-1}\cdot\sum_{j=1}^d \int_{\log(1-\phi^{(j)})}^\infty G\Big(x,\frac{1}{{\rm e}^x}\big(\boldsymbol{\phi}+({\rm e}^x-1)\boldsymbol{e}_j\big)\Big)\cdot ({\rm e}^x+2\phi^{(j)}-1)^{-(1+\alpha)}\cdot {\rm e}^x dx,
	\end{equation*}
where $\boldsymbol{\phi}=(\boldsymbol{\phi}^{(1)},\ldots,\boldsymbol{\phi}^{(d)})\in\mathcal{S}^{d,+}_{1}$. By the very definition of conditional expectation, to prove (ii) we must show that for every bounded measurable function $h:\mathcal{S}^{d,+}_{1}\to\mathbb{R}$ and $\boldsymbol{\theta}\in\mathcal{S}^{d,+}_{1}$,
\begin{equation}
\label{conditional_expectation_wts}
\mathbb{E}_{(0,\boldsymbol{\theta})}[H(\hat{\Xi}_{L-})h(\hat{\Xi}_{L-})]=\mathbb{E}_{(0,\boldsymbol{\theta})}[G(\Delta \hat{\xi}_{L},\hat{\Xi}_{L})h(\hat{\Xi}_{L-})].
\end{equation}
We begin by computing the LHS utilizing the compensation formula in effectively the same way we had done in our proof of Lemma~\ref{compact_thm_killing_rate} (and for this reason we skip many intermediate steps):
	\begin{align*}
&\mathbb{E}_{(0,\boldsymbol{\theta})}[H(\hat{\Xi}_{L-})h(\hat{\Xi}_{L-})] \\
&=E_{\boldsymbol{\theta}}\Big[T\big(\arg(X_{\tau^D-})\big)h\big(\arg(X_{\tau^D-})\big)\Big] \\
&=c\int_0^\infty \mathbb{E}_{(0,\boldsymbol{\theta})}\Big[\sum_{m=1}^d \int_{-\infty}^{-\Xi_v^{(m)}} H(\Xi_v)h(\Xi_v) |l|^{-(1+\alpha)}dl;v<K\Big] dv \\
&= c\int_0^\infty \mathbb{E}_{(0,\boldsymbol{\theta})}\Bigg[H(\Xi_v)h(\Xi_v)\Big(\sum_{m=1}^d\int_{\Xi_v^{(m)}}^\infty l^{-(1+\alpha)}dl\Big);v<K\Bigg] dv \\
&=\int_0^\infty \mathbb{E}_{(0,\boldsymbol{\theta})}\Big[H(\Xi_v)h(\Xi_v)\Big(\frac{c}{\alpha}\sum_{m=1}^d (\Xi_v^{(j)})^{-\alpha}\Big);v<K\Big] dv \\
&=\int_0^\infty \mathbb{E}_{(0,\boldsymbol{\theta})}\Big[H(\Xi_v)h(\Xi_v)q(\Xi_v);v<K\Big] dv.
\end{align*}
where we recall Lemma~\ref{compact_thm_killing_rate} with $\rho_1= \dots =\rho_d=1/2$ for the last equality.
By the way we had defined $H$, the $q(\Xi_v)$ in the above integrand cancels and we get
\begin{align*} 
&\mathbb{E}_{(0,\boldsymbol{\theta})}[H(\hat{\Xi}_{L-})h(\hat{\Xi}_{L-})] \\
&= \mathbb{E}_{(0,\boldsymbol{\theta})}\Bigg[\int_0^{K}\Bigg\{h(\Xi_v)\sum_{j=1}^d \int_{\log(1-\Xi_v^{(j)})}^\infty G\Big(x,\frac{1}{{\rm e}^x}\big(\Xi_v+({\rm e}^x-1)\boldsymbol{e}_j\big)\Big) ({\rm e}^x+2\Xi_v^{(j)}-1)^{-(1+\alpha)} {\rm e}^x dx\Bigg\}dv\Bigg].
\end{align*}
We will now work out the RHS of (\ref{conditional_expectation_wts}) and show that it equals the above expectation. Firstly,
	\begin{equation}
		\label{eq_3}
		\Delta \hat{\xi}_{L}=\log(\frac{\|\hat{X}_{\tau_D}\|_1}{\|\hat{X}_{\tau_D-}\|_1})=\log(\frac{\|R(X_{\tau_D})\|_1}{\|X_{\tau_D-}\|_1})=\log(\frac{\|X_{\tau_D}\|_1}{\|X_{\tau_D-}\|_1})%=\Delta\xi_\zeta,
	\end{equation}
	
	Secondly, we also need to derive an expression of $\hat{\Xi}_{L}$.% in terms of the MAP $(\xi,\Xi)$. 
	Using only the independence of all the $X^{(i)}$, we observe that on the event $\{\tau^D=(\tau^-_0)^{(j)}\}$, where $(\tau^-_0)^{(j)}\coloneqq\inf{\{t>0: X^{(j)}_t<0\}}$, i.e., on the event that the first time $X$ exits the orthant is through a jump in the $j$-th direction, we have
	\begin{equation}
		\label{eq_5}
		\hat{\Xi}_{L}=N^{(j)}\big(\arg(X_{\tau_D})\big)%N^{(j)}(\Xi_\zeta).
	\end{equation}
Therefore, by using the very definition of the MAP $(\xi,\Xi)$ and applying the usual compensation formula method that brought about Lemma~\ref{compact_thm_killing_rate} and Theorem~\ref{thm_quintuple_law_up_to_first_exit_l1_norm} (for which we skip many of the familiar intermediate steps and algebraic manipulations),
\begin{align*}
&\mathbb{E}_{(0,\boldsymbol{\theta})}\Big[G(\Delta \hat{\xi}_{L},\hat{\Xi}_{L})h(\hat{\Xi}_{L-})\Big] \\
%&=\sum_{j=1}^d \mathbb{E}_{(0,\boldsymbol{\theta})}\Big[G\big(\Delta \xi^*_{\zeta},N^{(j)}(\Xi^*_\zeta)\big)h(\hat{\Xi}_{L-})\mathbbm{1}_{\{\tau^D=(\tau^-_0)^{(j)}\}}\Big] \\
&=\sum_{j=1}^d E_{\boldsymbol{\theta}}\Bigg[G\Big(\log{\frac{\|X_{\tau_D}\|_1}{\|X_{\tau_D-}\|_1}},N^{(j)}\big(\arg(X_{\tau_D})\big)\Big)h\big(\arg(X_{\tau_D-})\big)\mathbbm{1}_{\{\tau_D=(\tau^-_0)^{(j)}\}}\Bigg] \\
&= \mathbb{E}_{(0,\boldsymbol{\theta})}\Bigg[\int_0^{K}\Bigg\{h(\Xi_v)\sum_{j=1}^d \int_{-\infty}^{-\Xi_v^{(j)}} G\Big(\log{\|\Xi_v+\ell\boldsymbol{e}_j\|_1},\frac{1}{\|\Xi_v+\ell\boldsymbol{e}_j\|_1}N^{(j)}(\Xi_v+\ell\boldsymbol{e}_j)\Big) c| \ell |^{-(1+\alpha)}dl\Bigg\}dv\Bigg].
\end{align*}
The normed expressions appearing in the integrand can then be further evaluated in exactly the same way that was done towards the end of the proof of Theorem~\ref{thm_quintuple_law_up_to_first_exit_l1_norm}, and so for that reason we skip these (already familiar) steps. By performing said calculations, one will indeed arrive at the desired LHS of (\ref{conditional_expectation_wts}) we had computed earlier.

{\color{black}Finally, to complete the proof,  part (ii) is a straightforward re-wording of part (i).} \qed
\end{proof}

\section{Skorokhod-reflected stable processes}
\label{sec_skorokhod_reflected_stable}

We investigate a different notion of reflection from that of Section~\ref{sec_reflected_symm_stable}: we study the Skorokhod-reflection of the process $X$ from (\ref{def_of_our_d_dim_stable}), which we denote by $\tilde{R}=(\tilde{R}_t)_{t\geq0}$. This process is defined by
\begin{equation}
\label{def_skorokhod_reflected_d_dim_stable}
\tilde{R}_t=(X_t^{(1)}-(0\land\underbar{X}_t^{(1)}),\ldots,X_t^{(d)}-(0\land\underbar{X}_t^{(d)})),\qquad t\geq0,
\end{equation}
where the $X^{(i)}$ are the $1$-dimensional stable processes from Section~\ref{sec_stable_killed_exit_cone} with $\alpha\in(1,2)$ and positivity parameter $\rho_i\equiv\rho$ satisfying $\alpha(1-\rho)=1$, so that each $X^{(i)}$ is spectrally-positive. We note that, besides the notion of reflection being completely different, a major difference between the reflected $d$-dimensional stable processes $\tilde{R}$ and $\hat{X}$ from the previous Section~\ref{sec_reflected_symm_stable} is that the former's reflection will occur in a smooth/continuous fashion; and we have already seen in Section~\ref{sec_reflected_symm_stable} that $\hat{X}$, instead, always makes a jump at its reflection time. Interested readers can find out more about this notion of reflection and its history in the original two papers of Skorokhod, \cite{skorokhod1961stochastic1}, \cite{skorokhod1962stochastic2}; as well as some later-written papers from various other authors; cf., \cite{tanaka1979stochastic}, \cite{kruk2007explicit}, \cite{dupuis1991lipschitz}. Nevertheless, it can be proved that $\tilde{R}$ is a ssMp. A proof (in dimension one) can be found in Lemma 6.19 of \cite{kyprianou2022stable}. It follows that the process $R=(R_t)_{t\geq0}$ defined by 
\begin{equation}
	\label{ssMp_from_item_c_specific}
R_t=\tilde{R}_t\mathbbm{1}_{\{t<\upsilon\}},\qquad t\geq0,
\end{equation}
 where $\upsilon=\inf\{s>0:\|\tilde{R}_s\|_1=0\}$, taking $\boldsymbol{0}_d\in\mathbb{R}^d$ as a cemetery state,  possess an underlying MAP, which (to be consistent with the introductory section) we denote by $(\xi^R,\Xi^R)=(\xi^R_t,\Xi^R_t)_{t\geq0}$. 

Our aim in this section is to characterize the above MAP. We should note that it is now no longer enough to solely describe its jump structure. Indeed, the nature of the kind of reflection we are studying in this section (happening in a continuous fashion), raises the necessity of a mathematical object that says something about $(\xi^R,\Xi^R)$'s behaviour near the boundary of $D=(0,\infty)^d$ as well, and in particular, the rate at which it "bounces" as it hits the boundary and gets reflected. To this end, the generator of $(\xi^R,\Xi^R)$, as well as all the (reflecting) boundary conditions satisfied by functions in its domain, are the most appropriate in fully characterizing it.

Before stating our results we need to introduce and explain some unconventional terminology and notation we employ in the remainder of this section, as well as Section~\ref{sec_reflected_BM}. Because the MAP we are concerned with takes values in the product space $\mathbb{R}\times\mathcal{S}_1^{d,+}$, which is a subset of $\mathbb{R}^{d+1}$, we need to come up with a more convenient/natural way of expressing the $(d+1)$-dimensional vector $$(y,\boldsymbol{\theta})=(y,\theta^{(1)},\ldots,\theta^{(d)})\in\mathbb{R}\times\mathcal{S},$$
through the orthonormal basis of $\mathbb{R}^{d+1}$. To this end, we shall call the first (traditional) coordinate thereof (containing the entry $y\in\mathbb{R}$) \textit{the $0$-th coordinate} of the vector; and, for $j>0$, the $(j+1)$-th (traditional) coordinate thereof (containing the entry $\theta^{(j)}$) \textit{the $j$-th coordinate} of the vector. This being the case, we define the $(d+1)$-dimensional vector $\boldsymbol{e}_0$ to be the vector 
\begin{equation*}
	\boldsymbol{e}_0\coloneqq (1,\underbrace{0,\ldots,0}_{\text{$d$-many}});
\end{equation*}
and the $(d+1)$-dimensional vector $\boldsymbol{e}_j$, for $j>0$, to be the vector
\begin{equation*}
	\boldsymbol{e}_j\coloneqq (0,\underbrace{0,\ldots,0,1,0,\ldots,0}_{\text{$d$-many coordinates}}),
\end{equation*}
where the entry $1$ is in the (traditional) $(j+1)$-th coordinate of the above vector. The collection of vectors $\{\boldsymbol{e}_0,\boldsymbol{e}_1,\ldots,\boldsymbol{e}_d\}$ is then, actually, nothing other than the standard orthonormal basis of $\mathbb{R}^{d+1}$.

We split the proof of Theorem~\ref{thm_generator_underlying_MAP_skorokhod_reflection} into two parts: we first derive the generator and reflecting boundary conditions of the original $d$-dimensional ssMp, $R$, via semimartingale stochastic integration theory, as per the well-known book by Protter, \cite{protter2005stochastic}. In fact, we obtain a more general result for $R$ involving, not the $X^{(i)}$ of this section, but more general $1$-dimensional spectrally-positive L\'evy processes. Once we have obtained the desired aforementioned objects for $R$, we then "convert" these, through basic calculus and the Volkonskii formula, to those $(\xi^R,\Xi^R)$ must possess. 

To this end, let us now, for the moment, redefine $X^{(i)}$, $1\leq i\leq d$, of this section to be the more general (independent) one-dimensional spectrally-positive L\'evy process with L\'evy-It\^o decomposition
$$X^{(i)}=X^{(i,1)}+X^{(i,2)}+X^{(i,3)},$$
where $X^{(i,1)}_t\coloneqq b_it+\sigma_i B^{(i)}_t$, with $b_i\in\mathbb{R}$, $\sigma_i>0$ and $B^{(i)}=(B^{(i)}_t)_{t\geq0}$ is a (one-dimensional) standard Brownian motion; $$X^{(i,2)}_t=\int_{(0,t]\times [1,\infty)} x N^{(i)}(ds,dx),$$
where $N^{(i)}$ denotes the Poisson random measure associated with the jumps of $X^{(i)}$, is the compound Poisson process responsible for the jumps of $X^{(i)}$ that are of size bigger than $1$ - independent of $X^{(i,1)}$; and $X^{(i,3)}$ is the square-integrable martingale responsible for the jumps of $X^{(i)}$ that are of size smaller than $1$ - independent of both $X^{(i,1)}$ and $X^{(i,2)}$. 

\begin{theorem}
	\label{thm_generator_of_skorokhod_reflected_d_dim_spectrally_positive_levy}
	Let $X=(X_t)_{t\geq 0}$, with $X_t=(X^{(1)}_t,\ldots,X_t^{(d)})$, where the $X^{(i)}$ are independent (standard) one-dimensional spectrally-positive L\'evy processes, as defined just above. Define the Skorokhod-reflection of $X$ by $R=(R^{(1)},\ldots,R^{(d)})$, where $R^{(i)}\coloneqq X^{(i)}-\underbar{X}^{(i)}$. Then, for $g\in C^2_b([0,\infty)^d)$ satisfying the boundary conditions
	\begin{equation*}
		\frac{\partial g}{\partial z_i}(z_1,\ldots,z_{i-1},0,z_{i+1},\ldots,z_d)=0\enspace \forall z_k>0,\qquad i=1,\ldots,d,
	\end{equation*}
	the process
	\begin{equation*}
		g(R_t)-g(R_0)-\int_0^\infty \mathcal{L}g(R_s) ds\text{ is a martingale},
	\end{equation*}
	where $\mathcal{L}:C^2_b([0,\infty)^d)\to C^2_b([0,\infty)^d)$ is the operator defined by
	\begin{equation*}
		\mathcal{L}F(\boldsymbol{w})=\sum_{i=1}^d \Bigg( b_i \frac{\partial F}{\partial z_i}(\boldsymbol{w}) +\frac{\sigma_i^2}{2}\frac{\partial^2 F}{\partial z_i^2}(\boldsymbol{w})+\int_0^\infty \Big(F(\boldsymbol{w}+x\boldsymbol{e}_i)-F(\boldsymbol{w})-\frac{\partial F}{\partial z_i}(\boldsymbol{w}) x\mathbbm{1}_{(0,1)}(x)\Big)\Pi^{(i)}(dx) \Bigg),
	\end{equation*}
	where $\Pi^{(i)}$ denotes the L\'evy measure of $X^{(i)}$.
\end{theorem}
\begin{proof}
The proof uses the same stochastic-calculus strategy using It\^o's formula (albeit the higher-dimensional variant) that brought about Proposition~5 of \cite{ramirez2024sticky}, which had to do with deriving a martingale associated with the $1$-dimensional sticky L\'evy process. Specifically, we first fix a function $g\in C^2_b([0,\infty)^d)$ and consider the $C^2_b(\mathbb{R}^{2d})$ function $f$ given by $$f(x_1,y_1,x_2,y_2,\ldots,x_d,y_d)=g(x_1-y_1,x_2-y_2,\ldots,x_d-y_d),\quad y_i\leq x_i.$$ We then apply Theorem 33, Chapter 2 of \cite{protter2005stochastic} to the (high-dimensional) semimartingale $$(X^{(1)},\underbar{X}^{(1)},X^{(2)},\underbar{X}^{(2)},\ldots,X^{(d)},\underbar{X}^{(d)})$$ with respect to the above-defined function $f$. By then adopting the aforementioned strategy involving semimartingale stochastic integration theory, we obtain the desired result. A full proof with more details can be found in the PhD thesis of one of the authors of this paper, HSM. \qed
\end{proof}

For some further martingales associated with the Skorokhod-reflection of a general spectrally-positive L\'evy process in dimension one we refer the reader to the paper \cite{nguyen2005some}, where there is an additional emphasis on their applications in fluctuation theory.

We are now in the position to prove Theorem~\ref{thm_generator_underlying_MAP_skorokhod_reflection}.

\begin{proof}[Proof of Theorem~\ref{thm_generator_underlying_MAP_skorokhod_reflection}]

\begin{comment}
	We first recall the definition of $(\xi^R,\Xi^R)$:
	\begin{equation*}
		(\xi^R_t,\Xi_t^R)=(\log{\|R_{I_t}\|},\arg(R_{I_t})),\qquad t<\varsigma,
	\end{equation*}
	where $I_t\coloneqq\inf\{s>0:\int_0^s \|R_u\|^{-\alpha} du>t\}$, for $t<\varsigma\coloneqq \int_0^\upsilon \|R_u\|^{-\alpha} du$, where we recall $\upsilon=\inf\{s>0:R_s=\boldsymbol{0}_d\}$.
	\end{comment}
	
	Because we are dealing with ($1$-dimensional) $\alpha$-stable $X^{(i)}$ - which we know do not possess a linear Brownian motion component in their L\'evy-It\^o decomposition - we take $b_i=\sigma_i=0$ in the statement of Theorem~\ref{thm_generator_of_skorokhod_reflected_d_dim_spectrally_positive_levy}. By Volkonskii's formula (cf., (21.4) of Chapter III.3, \cite{rogers2000diffusions}), for functions $g\in C^2_b([0,\infty)^d)$ satisfying the boundary conditions from the statement of Theorem~\ref{thm_generator_of_skorokhod_reflected_d_dim_spectrally_positive_levy}, the process
	\begin{equation*}
		g(R_{I_t})-g(R_{I_0})-\int_0^t \mathcal{A}g(R_{I_s}) ds,\qquad 0\leq t\leq \phi^R(\upsilon),
	\end{equation*}
where $\phi^R(t)=\int_0^t \|R_u\|_1^{-\alpha}$, $0\leq t<\upsilon$, and $I_t$ is the right-continuous inverse of $\phi^R$, is a martingale, where $\mathcal{A}:C^2_b([0,\infty)^d)\to C^2_b([0,\infty)^d)$ is the operator defined by
	\begin{equation*}
		\mathcal{A}g(\boldsymbol{w})=\|\boldsymbol{w}\|_1^\alpha \mathcal{L}g(\boldsymbol{w}),
	\end{equation*}
	where $\mathcal{L}$ is the operator from the statement of Theorem~\ref{thm_generator_of_skorokhod_reflected_d_dim_spectrally_positive_levy} (with $b_i=\sigma_i=0$, and $\Pi^{(i)}$ being the respective L\'evy measure of $X^{(i)}$ from the setting of this theorem).
	
	Let us now consider the $f\in C^2_b(\mathbb{R}\times\mathcal{S}^{d,+}_{1})$ from our theorem hypothesis, i.e., $f\in\mathcal{D}$. We define the following $C^2_b([0,\infty)^d)$ function $g$: 
	\begin{equation*}
		g:[0,\infty)^d\setminus\{\boldsymbol{0}_d\}\to\mathbb{R},\quad g(\boldsymbol{w})=f\big(\log{\|\boldsymbol{w}\|_1},\arg{(\boldsymbol{w})}\big).
	\end{equation*}
	We claim that $g$ satisfies the boundary conditions from the statement of Theorem~\ref{thm_generator_of_skorokhod_reflected_d_dim_spectrally_positive_levy}. Indeed, by the chain rule, for $\boldsymbol{w}=(w_1,\ldots,w_d)\in[0,\infty)^d\setminus\{\boldsymbol{0}_d\}$ and $i\in\{1,\ldots,d\}$,
	\begin{align*}
		\frac{\partial g}{\partial z_i}(\boldsymbol{w})&=\|\boldsymbol{w}\|_1^{-1}\frac{\partial}{\partial z_i}(\|\boldsymbol{w}\|_1)f_{y_0}(\log{\|\boldsymbol{w}\|_1},\arg(\boldsymbol{w}))-\frac{\partial}{\partial z_i}(\|\boldsymbol{w}\|_1)\sum_{j\neq i} \frac{w_j}{\|\boldsymbol{w}\|_1^2} f_{y_j}(\log{\|\boldsymbol{w}\|_1},\arg(\boldsymbol{w})) \\
		&\qquad+\Big(\frac{1}{\|\boldsymbol{w}\|_1}-\frac{\partial}{\partial z_i}(\|\boldsymbol{w}\|_1)\frac{w_i}{\|\boldsymbol{w}\|_1^2}\Big)f_{y_i}(\log{\|\boldsymbol{w}\|_1},\arg(\boldsymbol{w})) \\
		&=\|\boldsymbol{w}\|_1^{-1}\Big(\frac{\partial}{\partial z_i}(\|\boldsymbol{w}\|_1)f_{y_0}(\log{\|\boldsymbol{w}\|_1},\arg(\boldsymbol{w}))+f_{y_i}(\log{\|\boldsymbol{w}\|_1},\arg(\boldsymbol{w})) \\
		&\qquad-\frac{\partial}{\partial z_i}(\|\boldsymbol{w}\|_1)\sum_{j=1}^d \arg(\boldsymbol{w})^{(j)} f_{y_j}(\log{\|\boldsymbol{w}\|_1},\arg(\boldsymbol{w}))\Big).
	\end{align*}
	From our assumption that $f\in\mathcal{D}$, as well as the trivial fact that $\frac{\partial}{\partial z_i}(\|\boldsymbol{w}\|_1)=1$ for every $i$, it follows that $g$ satisfies the said boundary conditions. Finally, by writing the above operator $\mathcal{A}$ in the coordinate system given by $(x,\boldsymbol{\theta})=(\log{\|\boldsymbol{w}\|_1},\arg(\boldsymbol{w}))$, for $\boldsymbol{w}\in\mathbb{R}^d$, one obtains the desired operator $A$. \qed
\end{proof}
	
\section{Reflected Brownian motion}
\label{sec_reflected_BM}

In this section, using methods from the previous one and some results from the theory of stochastic differential equations of diffusions with reflecting boundary conditions, we obtain the generator of the underlying MAP of $d$-dimensional (Skorokhod-)reflected Brownian motion (a known ssMp with index of self-similarity $2$), as always, with respect to the norm $\|\cdot\|_{1}$. We also derive the SDE of which it is a unique weak solution of in the special case of dimension $d=2$ (which we had seen in Corollary~\ref{cor_sde_of_underlying_MAP_of_sk_reflected_bm_in_l1_d2}). 

We denote the above-mentioned ssMp, up to the time of its absorption at the zero vector, by $\mathcal{R}=(\mathcal{R}_t)_{t\geq0}$, 
\begin{equation}
	\label{ssMp_item_d_specific}
	\mathcal{R}_t\coloneqq (B_t^{(1)}-(0\land\underbar{B}_t^{(1)}),\ldots,B_t^{(d)}-(0\land\underbar{B}_t^{(d)}))\mathbbm{1}_{\{t<\upsilon\}},\qquad \upsilon\coloneqq\inf\{t>0:\mathcal{R}_t=\boldsymbol{0}_d\},
\end{equation}
where the $B^{(i)}$ are independent one-dimensional Brownian motions. We denote the underlying MAP of $\mathcal{R}$ with respect to the $L^1$ norm by $(\rho,\Theta)=(\rho_t,\Theta_t)_{t\geq0}$. 

\begin{theorem}
	\label{thm_lp_generator_of_MAP_of_SK_reflected_d_dim_BM}
 Let $(\rho,\Theta)$ be the underlying MAP of $\mathcal{R}$ from (\ref{ssMp_item_d_specific}) with respect to the norm  $\|\cdot\|_1$. Then, the operator $A:\mathcal{D}\to C^2_b(\mathbb{R}\times\mathcal{S}^{d,+}_{1})$, where $\mathcal{D}$ is the collection of functions from (\ref{def:8zillion}), is given by
	\begin{align*}
		Af(x,\boldsymbol{\theta})&= \sum_{i=0}^d b_i(\boldsymbol{\theta})\frac{\partial f}{\partial y_i}(x,\boldsymbol{\theta}) + \frac{1}{2}\sum_{i,j=0}^d a_{ij}(\boldsymbol{\theta})\frac{\partial^2 f}{\partial y_i\partial y_j}(x,\boldsymbol{\theta}),
	\end{align*}
	where for $\boldsymbol{\theta}=(\theta_1,\ldots,\theta_d)\in\mathcal{S}^{d,+}_{1}$ and  $1\leq i\leq d$,
	\begin{equation*}
		b_0(\boldsymbol{\theta})=-\frac{d}{2},\quad b_i(\boldsymbol{\theta})=d\theta_i-1,\quad a_{00}(\boldsymbol{\theta})= d,\quad a_{ii}(\boldsymbol{\theta})=(1-\theta_i)^2+(d-1)\theta_i^2,\quad a_{0i}(\boldsymbol{\theta})=a_{i0}(\boldsymbol{\theta})= 1-d\theta_i,
	\end{equation*}
	
	\begin{equation*}
		a_{ij}(\boldsymbol{\theta})=a_{ji}(\boldsymbol{\theta})= d\theta_i\theta_j-\theta_i-\theta_j,\qquad i,j\in\{1,\ldots,d\},\enspace i\neq j,
	\end{equation*}
solves the martingale problem for $(\rho,\Theta)$.
\end{theorem}
\begin{proof}
	By plugging in $\sigma_i=1$, $b_i=0$, and $\Pi^{(i)}\equiv 0$ in Theorem~\ref{thm_generator_of_skorokhod_reflected_d_dim_spectrally_positive_levy}, we see that for functions $g\in C^2_b([0,\infty)^d)$ satisfying the boundary conditions of the aforementioned theorem, the process 
	\begin{equation*}
		g(\mathcal{R}_t)-g(\mathcal{R}_0)-\int_0^t \Big(\frac{1}{2}\sum_{i=1}^d \frac{\partial^2 g}{\partial z_i^2}(\mathcal{R}_s)\Big) ds,\qquad 0\leq t \leq\upsilon,
	\end{equation*}
	is a martingale, where $\upsilon$ was defined in (\ref{ssMp_item_d_specific}).
	
	Using the fact that $\mathcal{R}$ has index of self-similarity $2$ (since Brownian motion does), Volkonskii's formula tells us, in the same way it did in the proof of Theorem~\ref{thm_generator_underlying_MAP_skorokhod_reflection}, that for the above function $g$, the process 
	\begin{equation}
		\label{eq_volk_sk_reflect_BM_d_dim_MAP}
		g(\mathcal{R}_{I_t})-g(\mathcal{R}_0)-\int_0^t \Bigg(\frac{\|\mathcal{R}_{I_s}\|_1^{2}}{2}\sum_{i=1}^d \frac{\partial^2 g}{\partial z_i^2}(\mathcal{R}_{I_s})\Bigg) ds,\qquad t\geq0,
	\end{equation}
	is a martingale, where $I_t$ is the right-continuous inverse of $t\mapsto \int_0^t \|\mathcal{R}_s\|^{-2} ds$. We recall that $(\rho,\Theta)$ is the following space-time transformation of $\mathcal{R}$:
	\begin{equation*}
		(\rho_t,\Theta_t)=(\log{\|\mathcal{R}_{I_t}\|_1},\arg(\mathcal{R}_{I_t})),\qquad t\geq0.
	\end{equation*}
	Consider a $C^2_b(\mathbb{R}\times\mathcal{S}^{d,+}_{1})$ function $f$ from class $\mathcal{D}$. By plugging in the $C^2_b([0,\infty)^d)$ function $g(\boldsymbol{w})=f(\log{\|\boldsymbol{w}\|_1},\arg(\boldsymbol{w}))$ in (\ref{eq_volk_sk_reflect_BM_d_dim_MAP}) and rewriting the expression in the coordinate system given by
	\begin{equation*}
		(x,\boldsymbol{\theta})=(\log{\|\boldsymbol{w}\|_1},\arg(\boldsymbol{w})),\qquad \boldsymbol{w}\in [0,\infty)^d,
	\end{equation*}
	the result will follow. Indeed, our first step is to establish that, under our assumption that $f\in\mathcal{D}$, the above-defined function $g$ satisfies the boundary conditions of Theorem~\ref{thm_generator_of_skorokhod_reflected_d_dim_spectrally_positive_levy}. This, however, has already been done in the proof of Theorem~\ref{thm_generator_underlying_MAP_skorokhod_reflection}. Again in the proof of Theorem~\ref{thm_generator_underlying_MAP_skorokhod_reflection}, we had seen that for $1\leq i\leq d$,
	\begin{align*}
		\frac{\partial g}{\partial z_i}(\boldsymbol{w})&=\|\boldsymbol{w}\|_1^{-1}\Big(f_{y_0}(\log{\|\boldsymbol{w}\|_1},\arg(\boldsymbol{w}))+(1-\arg(\boldsymbol{w})^{(i)})f_{y_i}(\log{\|\boldsymbol{w}\|_1},\arg(\boldsymbol{w})) \\
		&\qquad-\sum_{j\neq i} \arg(\boldsymbol{w})^{(j)} f_{y_j}(\log{\|\boldsymbol{w}\|_1},\arg(\boldsymbol{w}))\Big).
	\end{align*}
	By differentiating once more with respect to the $z_i$ coordinate with the help of the product/quotient and chain rules from calculus, and then rewriting the resulting expression in our other coordinate system, we obtain $\frac{\partial^2 g}{\partial z_i^2}(\boldsymbol{w})$ thus:
	\begin{align*}
		{\rm e}^{-2x}\Big[& -f_{y_0}(x,\boldsymbol{\theta}) + 2(\theta_i-1)f_{y_i}(x,\boldsymbol{\theta}) + 2\sum_{j\neq i} \theta_j f_{y_j}(x,\boldsymbol{\theta}) + f_{y_0^2}(x,\boldsymbol{\theta}) + (\theta_i-1)^2f_{y_i^2}(x,\boldsymbol{\theta}) \\
		&+ \sum_{j\neq i} \theta_j^2f_{y_j^2}(x,\boldsymbol{\theta}) + 2(1-\theta_i)f_{y_0y_i}(x,\boldsymbol{\theta}) -2\sum_{j\neq i} \theta_j f_{y_0y_j}(x,\boldsymbol{\theta}) + \sum_{\substack{k,j\neq i \\ k\neq j}} \theta_j\theta_k f_{y_jy_k}(x,\boldsymbol{\theta}) \\
		&+ 2(\theta_i-1)\sum_{j\neq i} \theta_j f_{y_iy_j}(x,\boldsymbol{\theta}) \Big].
	\end{align*}
	The final step is then, using the above-derived expression, to write 
	\begin{equation*}
		\frac{\|\boldsymbol{w}\|_1^{2}}{2}\sum_{i=1}^d \frac{\partial^2 g}{\partial z_i^2}(\boldsymbol{w}),\qquad \boldsymbol{w}\in[0,\infty)^d,
	\end{equation*}
	in the other coordinate system. Since $\|\boldsymbol{w}\|_1^{2}$ is none other than ${\rm e}^{2x}$ in said coordinate system, summing the above-derived expression over all $i=1,\ldots,d$ and multiplying by $\frac{{\rm e}^{2x}}{2}$ yields the operator $A$ from the statement of the theorem. \qed
\end{proof}

We now follow the recipe from \cite{pilipenko2014introduction} to derive the SDE which the MAP from Theorem~\ref{thm_lp_generator_of_MAP_of_SK_reflected_d_dim_BM} above with $d=2$ (weakly) solves. We recall that this was given in Corollary~\ref{cor_sde_of_underlying_MAP_of_sk_reflected_bm_in_l1_d2}.

\begin{proof}[Proof of Corollary~\ref{cor_sde_of_underlying_MAP_of_sk_reflected_bm_in_l1_d2}]
	Consistent with the notation and terminology of Chapter 2.2 of \cite{pilipenko2014introduction}, what we are dealing with here is a particular example of the Skorokhod problem in the ($3$-dimensional) domain $D=\mathbb{R}\times\mathcal{S}^{2,+}_1\subseteq \mathbb{R}^3$, which has boundary $\partial D=\mathbb{R}\times\{(0,1),(1,0)\}$, and the following reflecting directions:
	\begin{equation}
		\label{eq_reflecting_sets}
		K_{(w,0,1)^T}=\begin{pmatrix}
			1 \\ 1 \\ -1
		\end{pmatrix}\text{ and } K_{(w,1,0)^T}=\begin{pmatrix}
			1 \\ -1 \\ 1
		\end{pmatrix},\qquad w\in\mathbb{R}.
	\end{equation}
	%as described by i and ii of corollary~\ref{cor_generator_of_underlying_MAP_of_sk_reflected_bm_in_l1_d2}. 
	The two equations from (\ref{eq_reflecting_sets}) plainly say that as the modulator $\Theta$ traverses continuously (i.e., without jumping) in the interior of the simplex in the first quadrant and hits either boundary point $(0,1)$ or $(1,0)$, it immediately bounces, again in a continuous fashion, back into the interior of the simplex, from which it came.
	
	We now follow the recipe illustrated at the beginning of Chapter 3.2 of \cite{pilipenko2014introduction}. More precisely, we first observe that, in our case, the analogous operator $L$ from (3.21) therein is the operator $A$ presented in our Theorem~\ref{thm_lp_generator_of_MAP_of_SK_reflected_d_dim_BM} with $d=2$, with the analogous $a_i$ and $\sigma_{ij}$ being, for $\boldsymbol{\theta}=(\theta_1,\theta_2)\in\mathcal{S}^{2,+}_{1}$,
	\begin{equation*}
		a_0(\boldsymbol{\theta})=-1,\qquad a_1(\boldsymbol{\theta})=\theta_1-\theta_2,\qquad a_2(\boldsymbol{\theta})=\theta_2-\theta_1 ;
	\end{equation*}
	and
	\begin{align*}
		&\sigma_{00}(\boldsymbol{\theta})=2,\qquad \sigma_{11}(\boldsymbol{\theta})=\theta_1^2+\theta_2^2,\qquad \sigma_{22}(\boldsymbol{\theta})=\theta_1^2+\theta_2^2,\qquad \sigma_{01}(\boldsymbol{\theta})=\sigma_{10}(\boldsymbol{\theta})=\theta_2-\theta_1,\\ &\sigma_{02}(\boldsymbol{\theta})=\sigma_{20}(\boldsymbol{\theta})=\theta_1-\theta_2,\qquad \sigma_{21}(\boldsymbol{\theta})=\sigma_{12}(\boldsymbol{\theta})=-(\theta_1^2+\theta_2^2);
	\end{align*}
	and the analogous operator $J$ from (3.20) of \cite{pilipenko2014introduction} that describes the reflecting boundary conditions is, in our case, the operator $J:C^2_b(\mathbb{R}\times\mathcal{S}^{2,+}_{1})\to C^2_b(\mathbb{R}\times\mathcal{S}^{2,+}_{1})$, supported on the boundary, given by 
	\begin{equation*}
		Jf(x,\boldsymbol{\theta})=\begin{cases*}
			(1,1,-1)\cdot\grad f(x,\boldsymbol{\theta}),&\text{if $(x,\boldsymbol{\theta})\in\mathbb{R}\times\{(0,1)\},$} \\
			(1,-1,1)\cdot\grad  f(x,\boldsymbol{\theta}),&\text{if $(x,\boldsymbol{\theta})\in\mathbb{R}\times\{(1,0)\};$}
		\end{cases*}
	\end{equation*}
	and therefore, the analogous $\gamma_i$ are, for $\boldsymbol{\theta}=(\theta_1,\theta_2)\in\mathcal{S}^{2,+}_{1}$,
	\begin{equation*}
		\gamma_0(\boldsymbol{\theta})=1,\qquad \gamma_1(\boldsymbol{\theta})=\mathbbm{1}_{\{(0,1)\}}(\boldsymbol{\theta})-\mathbbm{1}_{\{(1,0)\}}(\boldsymbol{\theta}),\qquad \gamma_2(\boldsymbol{\theta})=\mathbbm{1}_{\{(1,0)\}}(\boldsymbol{\theta})-\mathbbm{1}_{\{(0,1)\}}(\boldsymbol{\theta}).
	\end{equation*}
	Our Theorem~\ref{thm_lp_generator_of_MAP_of_SK_reflected_d_dim_BM} states that, for $f\in C^2_b(\mathbb{R}\times\mathcal{S}^{2,+}_{1})$ such that $Jf(x,\boldsymbol{\theta})=0$, $(x,\boldsymbol{\theta})\in\partial D$, the process
	\begin{equation*}
		f(\rho_t,\Theta_t)-f(\rho_0,\Theta_0)-\int_0^t Af(\rho_s,\Theta_s) ds,
	\end{equation*}
	where $t<\varsigma\coloneqq\inf\{s>0:\Theta_s=\partial\}=\int_0^\upsilon \|\mathcal{R}_s\|_{\text{L}^1}^{-2} ds$, is a martingale. In other words, the analogous (3.21) from \cite{pilipenko2014introduction} holds for every $f\in C^2_b(\mathbb{R}\times\mathcal{S}^{2,+}_{1})$ satisfying our analogous (3.20) from the same piece of text. This ensures $(\rho,\Theta)$ as a weak solution to the following (reflecting) SDE:
	\begin{equation}
		\label{sde_with_3_by_3_diffusion_matrix}
		d\begin{pmatrix}
			\rho_t \\ \Theta_t
		\end{pmatrix}=\boldsymbol{a}(\Theta_t) dt + \boldsymbol{b}(\Theta_t) d\boldsymbol{W}_t + \boldsymbol{\gamma}(\Theta_t) dl_t,
	\end{equation}
	where, for $\boldsymbol{\theta}\in\mathcal{S}^{2,+}_{1}$, $\boldsymbol{b}(\boldsymbol{\theta})$ is a matrix $(b_{ij}(\boldsymbol{\theta}))_{i,j=0}^2$ satisfying
	\begin{equation}
		\label{eq_b_and_sigma_relation}
		\sigma_{ij}(\boldsymbol{\theta})=\sum_{k=0}^2 b_{ki}(\boldsymbol{\theta}) b_{kj}(\boldsymbol{\theta})
	\end{equation}
	for every $i,j\in\{0,1,2\}$; $\boldsymbol{a}(\boldsymbol{\theta})$ is the vector $(a_0(\boldsymbol{\theta}),a_1(\boldsymbol{\theta}),a_2(\boldsymbol{\theta}))$; $\boldsymbol{\gamma}(\boldsymbol{\theta})$ is the vector $(\gamma_0(\boldsymbol{\theta}),\gamma_1(\boldsymbol{\theta}),\gamma_2(\boldsymbol{\theta}))$; $\boldsymbol{W}$ is $3$-dimensional Brownian motion, and $l$ is local time of $\Theta$ at the boundary $\partial D$.
	
	As is explained in Theorem 2.2 of \cite{stroock1972support} (albeit in the setting where there are no boundary conditions present, however the same argument still holds even when there are), one can simply choose the matrix $\boldsymbol{b}(\boldsymbol{\theta})$ to be the unique positive-definite symmetric square-root of the matrix $\boldsymbol{\sigma}(\boldsymbol{\theta})=(\sigma_{ij}(\boldsymbol{\theta}))_{i,j=0}^2$ - and this is what we do below. After many (but elementary) linear algebra computations, utilizing the identity $\theta_1+\theta_2=1$, one will find: 
	
	\begin{equation*}
		\boldsymbol{b}(\boldsymbol{\theta})=\begin{pmatrix}
			\frac{2\theta_1\theta_2}{\sqrt{\lambda_2}+\sqrt{\lambda_3}}+\frac{\sqrt{\lambda_2}+\sqrt{\lambda_3}}{2} & -\frac{\theta_1-\theta_2}{\sqrt{\lambda_2}+\sqrt{\lambda_3}} & \frac{\theta_1-\theta_2}{\sqrt{\lambda_2}+\sqrt{\lambda_3}} \\
			-\frac{\theta_1-\theta_2}{\sqrt{\lambda_2}+\sqrt{\lambda_3}} & \frac{\sqrt{\lambda_2}+\sqrt{\lambda_3}}{4}-\frac{\theta_1\theta_2}{\sqrt{\lambda_2}+\sqrt{\lambda_3}} & \frac{\theta_1\theta_2}{\sqrt{\lambda_2}+\sqrt{\lambda_3}} - \frac{\sqrt{\lambda_2}+\sqrt{\lambda_3}}{4} \\
			\frac{\theta_1-\theta_2}{\sqrt{\lambda_2}+\sqrt{\lambda_3}} & \frac{\theta_1\theta_2}{\sqrt{\lambda_2}+\sqrt{\lambda_3}} - \frac{\sqrt{\lambda_2}+\sqrt{\lambda_3}}{4} & \frac{\sqrt{\lambda_2}+\sqrt{\lambda_3}}{4} - \frac{\theta_1\theta_2}{\sqrt{\lambda_2}+\sqrt{\lambda_3}}
		\end{pmatrix},
	\end{equation*}
	where the functions $\lambda_2=\lambda_2(\boldsymbol{\theta})$ and $\lambda_3=\lambda_3(\boldsymbol{\theta})$ are from the statement of the corollary. We now define a new ($2$-dimensional) Brownian motion, $\boldsymbol{\tilde{W}}=(\boldsymbol{\tilde{W}}^{(1)},\boldsymbol{\tilde{W}}^{(2)})$, in terms of our $3$-dimensional Brownian motion, $\boldsymbol{W}=(\boldsymbol{W}^{(1)},\boldsymbol{W}^{(2)},\boldsymbol{W}^{(3)})$, by setting $\boldsymbol{\tilde{W}}^{(1)}\coloneqq \boldsymbol{W}^{(1)}$ and $\boldsymbol{\tilde{W}}^{(2)}\coloneqq \boldsymbol{W}^{(2)}-\boldsymbol{W}^{(3)}$. Finally, observe how we can rewrite the SDE from (\ref{sde_with_3_by_3_diffusion_matrix}) in the following way:
	\begin{equation*}
		d\begin{pmatrix}
			\rho_t \\ \Theta_t
		\end{pmatrix}=\boldsymbol{a}(\Theta_t) dt + \boldsymbol{\tilde{b}}(\Theta_t) d\boldsymbol{\tilde{W}}_t + \boldsymbol{\gamma}(\Theta_t) dl_t,
	\end{equation*}
	where
	\begin{equation*}
		\boldsymbol{\tilde{b}}(\boldsymbol{\theta})\coloneqq\begin{pmatrix}
			\frac{2\theta_1\theta_2}{\sqrt{\lambda_2}+\sqrt{\lambda_3}}+\frac{\sqrt{\lambda_2}+\sqrt{\lambda_3}}{2} & -\frac{\theta_1-\theta_2}{\sqrt{\lambda_2}+\sqrt{\lambda_3}} \\
			-\frac{\theta_1-\theta_2}{\sqrt{\lambda_2}+\sqrt{\lambda_3}} & \frac{\sqrt{\lambda_2}+\sqrt{\lambda_3}}{4}-\frac{\theta_1\theta_2}{\sqrt{\lambda_2}+\sqrt{\lambda_3}} \\
			\frac{\theta_1-\theta_2}{\sqrt{\lambda_2}+\sqrt{\lambda_3}} & \frac{\theta_1\theta_2}{\sqrt{\lambda_2}+\sqrt{\lambda_3}} - \frac{\sqrt{\lambda_2}+\sqrt{\lambda_3}}{4}
		\end{pmatrix}.
	\end{equation*}
	And so we discern that the source of randomness of the SDE our MAP solves is actually two-dimensional. As for the uniqueness of the solution to the SDE, it is not difficult to check that the required conditions detailed in Section 5 of \cite{stroock1971diffusion}; or, alternatively, those of Theorem 3.2.3 or Theorem 2.1.1 from \cite{pilipenko2014introduction} (as remarked in Exercise 2.2.1 of \cite{pilipenko2014introduction}, the conditions for uniqueness listed in Theorem 2.1.1 therein also hold in our setup with domain $D$) are satisfied by our setup.  \qed
\end{proof}

\section*{Acknowledgements}
HSM is supported by the Mathematics and Statistics Centre for Doctoral Training, and gratefully acknowledges funding from the University of Warwick. Part of this work was conducted while VR was visiting  the Department of Statistics at the University of Warwick, United Kingdom; he would like to thank his hosts for partial financial support as well as for their kindness and hospitality. In addition, VR is grateful for additional financial support from CONAHCyT-Mexico, grant nr. 852367.

	\bibliographystyle{plain}
\bibliography{norm_dependent_lamperti_type_connection_ssMps_MAPs_orthant_bib.bib}

\begin{thebibliography}{10}

\bibitem{alili2017inversion}
Larbi Alili, Lo{\"\i}c Chaumont, Piotr Graczyk, and Tomasz {\.Z}ak.
\newblock Inversion, duality and Doob h-transforms for self-similar Markov
  processes.
\newblock 2017.

\bibitem{blumenthal2007markov}
R.M.C. Blumenthal and R.K. Getoor.
\newblock {\em Markov Processes and Potential Theory}.
\newblock Dover books on mathematics. Dover Publications, 2007.

\bibitem{caballero2006conditioned}
Maria~Emilia Caballero and Lo{\"\i}c Chaumont.
\newblock Conditioned stable L{\'e}vy processes and the Lamperti
  representation.
\newblock {\em Journal of Applied Probability}, 43(4):967--983, 2006.

\bibitem{chaumont2006introduction}
L~Chaumont.
\newblock Introduction aux processus auto-similaires.
\newblock {\em Cours du DEA de Proabilit{\'e}s et Applications. Universit{\'e}
  de Paris VI}, 2006.

\bibitem{chaumont2013lamperti}
Lo{\"\i}c Chaumont, Henry Pant{\'\i}, and V{\'\i}ctor Rivero.
\newblock The Lamperti representation of real-valued self-similar Markov
  processes.
\newblock 2013.

\bibitem{ccinlar1972markovI}
Erhan {\c{C}}inlar.
\newblock Markov additive processes. I.
\newblock {\em Zeitschrift f{\"u}r Wahrscheinlichkeitstheorie und verwandte
  Gebiete}, 24(2):85--93, 1972.

\bibitem{ccinlar1972markovII}
Erhan {\c{C}}inlar.
\newblock Markov additive processes. II.
\newblock {\em Zeitschrift f{\"u}r Wahrscheinlichkeitstheorie und verwandte
  Gebiete}, 24(2):95--121, 1972.

\bibitem{ccinlar1973levy}
Erhan {\c{C}}inlar.
\newblock L{\'e}vy systems of Markov additive processes.
\newblock Technical report, Discussion Paper, 1973.

\bibitem{cinlar1975exceptional}
Erhan Cinlar.
\newblock Exceptional paper—Markov renewal theory: A survey.
\newblock {\em Management Science}, 21(7):727--752, 1975.

\bibitem{ccinlar1976entrance}
Erhan {\c{C}}inlar.
\newblock Entrance-exit distributions for Markov additive processes.
\newblock {\em Stochastic Systems: Modeling, Identification and Optimization,
  I}, pages 22--38, 1976.

\bibitem{dupuis1991lipschitz}
Paul Dupuis and Hitoshi Ishii.
\newblock On Lipschitz continuity of the solution mapping to the Skorokhod
  problem, with applications.
\newblock {\em Stochastics: An International Journal of Probability and
  Stochastic Processes}, 35(1):31--62, 1991.

\bibitem{kallenberg1997foundations}
Olav Kallenberg.
\newblock {\em Foundations of modern probability}, volume~2.
\newblock Springer.

\bibitem{ken1999levy}
S.~Ken-Iti.
\newblock {\em L{\'e}vy Processes and Infinitely Divisible Distributions}.
\newblock Cambridge studies in advanced mathematics. Cambridge University
  Press, 1999.

\bibitem{kruk2007explicit}
Lukasz Kruk, John Lehoczky, Kavita Ramanan, and Steven Shreve.
\newblock An explicit formula for the Skorokhod map on [0, a].
\newblock 2007.

\bibitem{kyprianou2014fluctuations}
A.E. Kyprianou.
\newblock {\em Fluctuations of L{\'e}vy Processes with Applications:
  Introductory Lectures}.
\newblock Universitext. Springer Berlin Heidelberg, 2014.

\bibitem{kyprianou2022stable}
A.E. Kyprianou and J.C. Pardo.
\newblock {\em Stable L{\'e}vy Processes via Lamperti-Type Representations}.
\newblock IMS monographs. Cambridge University Press, 2022.

\bibitem{kyprianou2014hitting}
Andreas~E Kyprianou, Juan~Carlos Pardo, and Alexander~R Watson.
\newblock Hitting distributions of $\alpha$-stable processes via path censoring
  and self-similarity.
\newblock{\em The Annals of Probability} Vol. 42(1), pp. 398-
430  
\newblock 2014.

\bibitem{kyprianou2025stronglawlargenumbers}
Andreas~E. Kyprianou and Victor Rivero.
\newblock The strong law of large numbers and a functional central limit
  theorem for general Markov additive processes.
  \newblock {\em arXiv preprint arXiv:2505.10956} 
  2025.

\bibitem{lamperti1972semi}
John Lamperti.
\newblock Semi-stable Markov processes. I.
\newblock {\em Zeitschrift f{\"u}r Wahrscheinlichkeitstheorie und verwandte
  Gebiete}, 22(3):205--225, 1972.

\bibitem{nguyen2005some}
Laurent Nguyen-Ngoc and Marc Yor.
\newblock Some martingales associated to reflected L{\'e}vy processes.
\newblock {\em S{\'e}minaire de Probabilit{\'e}s XXXVIII}, pages 42--69, 2005.

\bibitem{pilipenko2014introduction}
Andrey Pilipenko.
\newblock {\em An introduction to stochastic differential equations with
  reflection}, volume~1.
\newblock Universit{\"a}tsverlag Potsdam, 2014.

\bibitem{protter2005stochastic}
Philip~E Protter and Philip~E Protter.
\newblock {\em Stochastic differential equations}.
\newblock Springer, 2005.

\bibitem{ramirez2024sticky}
Miriam Ramirez and Ger{\'o}nimo~Uribe Bravo.
\newblock The sticky Levy process as a solution to a time change equation.
\newblock {\em Journal of Mathematical Analysis and Applications},
  530(1):127742, 2024.

\bibitem{rogers2000diffusions}
Leonard~CG Rogers and David Williams.
\newblock {\em Diffusions, Markov processes, and Martingales: Volume 1,
  foundations}.
\newblock Cambridge university press, 2000.

\bibitem{siri2024lamperti}
Arno Siri-J{\'e}gousse and Alejandro~Hern{\'a}ndez Wences.
\newblock The Lamperti transformation in the infinite-dimensional setting and
  the genealogies of self-similar Markov processes.
\newblock {\em arXiv preprint arXiv:2405.10193}, 2024.

\bibitem{skorokhod1961stochastic1}
Anatoliy~V Skorokhod.
\newblock Stochastic equations for diffusion processes in a bounded region.
\newblock {\em Theory of Probability \& Its Applications}, 6(3):264--274, 1961.

\bibitem{skorokhod1962stochastic2}
Anatoliy~V Skorokhod.
\newblock Stochastic equations for diffusion processes in a bounded region. II.
\newblock {\em Theory of Probability \& Its Applications}, 7(1):3--23, 1962.

\bibitem{stroock1971diffusion}
Daniel~W Stroock and SR~Srinivasa Varadhan.
\newblock Diffusion processes with boundary conditions.
\newblock {\em Communications on Pure and Applied Mathematics}, 24(2):147--225,
  1971.

\bibitem{stroock1972support}
Daniel~W Stroock and Srinivasa~RS Varadhan.
\newblock On the support of diffusion processes with applications to the strong
  maximum principle.
\newblock In {\em Proceedings of the Sixth Berkeley Symposium on Mathematical
  Statistics and Probability (Univ. California, Berkeley, Calif., 1970/1971)},
  volume~3, pages 333--359, 1972.

\bibitem{tanaka1979stochastic}
Hiroshi Tanaka.
\newblock Stochastic differential equations with reflecting.
\newblock {\em Stochastic Processes: Selected Papers of Hiroshi Tanaka}, 9:157,
  1979.

\end{thebibliography}
\end{document}